\documentclass[12pt,reqno]{amsart}
\usepackage{graphicx,hyperref}
\usepackage{amsthm,amsfonts,latexsym,amssymb,cite}
\newtheorem{lemma}{\bf{Lemma} }[section]
\newtheorem{prop}{\bf{Proposition}}[section]
\newtheorem{thm}{\bf{Theorem}}[section]
\newtheorem{rem}{\sc{Remark} }[section]
\newtheorem{dfn}{\sc{Definition} }[section]

\def\build#1_#2{\mathrel{\mathop{\kern 0pt#1}\limits_{#2}}}
  
\begin{document}

 \title[Explicit higher regularity on a mixed problem]{Explicit higher regularity on a Cauchy problem with  mixed Neumann-power
type boundary conditions}

\author{Luisa Consiglieri}
\address{Luisa Consiglieri, Independent Researcher Professor, European Union}
\urladdr{\href{http://sites.google.com/site/luisaconsiglieri}{http://sites.google.com/site/luisaconsiglieri}}

\begin{abstract}
We investigate the regularity in $L^p$ ($p>2$) of the gradient
of any weak solution of a Cauchy problem with
mixed Neumann-power type boundary conditions.
Under suitable assumptions we
prove the existence of weak solutions that satisfy explicit estimates.
Some considerations on the steady-state regularity are discussed.
\end{abstract}
\keywords{higher regularity, Caccioppoli estimate,
Gehring-Giaquinta-Modica theory,  $L^p$ maximal regularity}

\subjclass[2010]{35K20, 49N60}

\maketitle

\section{Introduction}

In the mathematical literature, the dependence on the data is
commonly hidden on the universal constants.
These constants that are involved in the estimates
are systematically assumed abstract, {\em i.e.}
they may change their numerical value from line to line throughout
the whole study in concern.
Our objective is to find explicit estimates (also known as 
quantitative estimates \cite{BGV}) such that allow
its real and true application to other fields of science.

In the study of the regularity on the initial-boundary value problem for the
second order differential equation in divergence form, at least
three shortcomings appear from the real world applications.
They are namely discontinuous leading coefficient,
nonlinear monotone boundary conditions, and nonsmooth
Lipschitz domain.
One of the approaches in the investigation of regularity is based on 
the difference quotient technique. We refer to \cite{ebm,fm,nw} 
where there are no boundary terms.  
The elliptic regularity in the halfspace can be found in \cite{kass}.
For Neumann-type boundary conditions,
an arbitrary bounded domain is not globally invariant with respect
to translations. The difference quotient technique is only allowed
 by a suitable localization procedure \cite{sav}.
Even the interior regularity requires the 
differentiability of coefficient,
which is not fulfilled by our coefficient.
The realization of the Laplace operator with
generalized nonlinear Robin boundary conditions can be found in \cite{bie}.

Also by the localization method,
the higher regularity of the gradient is obtained via the reverse
H\"older inequality with increasing supports 
(known as Gehring-Giaquinta-Modica theory,
cf. \cite{ark95,ark99,haga,nwff,par} and the references therein).
Here, we adopt this approach to determine explicit estimates for the
Cauchy problem inspired in the nonlinear heat equation
with the Neumann condition on one part of the boundary of the domain,
and the power law condition on the remaining part of the boundary that
includes the radiative effects \cite{lap,druet}.
Also the constants involved in $L^{p,\infty}$-estimate are determined.

Some considerations on the steady-state case are discussed in
Section \ref{sss}.

\section{Maximal parabolic regularity on $X$}

 Let $[0, T] \subset {\mathbb R}$ be the time interval with $ T >0$,
and   $\Omega\subset \mathbb{R}^n$ ($n\geq 2$) be a (bounded)  domain
 of class $C^{1}$. 
The boundary $\partial\Omega$ is decomposed into two 
 disjoint open subsets, namely $\Gamma$ and  $\partial\Omega\setminus
\bar \Gamma$.
 Moreover we set $Q_T=\Omega\times ]0,T[$, and $\Sigma_T=\Gamma\times ]0,T[$.

In the presence of  Lebesgue,
Sobolev, and Bochner spaces,
the functional framework is
\begin{align*}
L^{p,\infty}(Q_T)&=L^\infty(0,T;L^p(\Omega));\\
V_{p,\ell}(\Omega)&=\{v\in W^{1,p}(\Omega):\ 
v|_\Gamma\in L^{\ell}(\Gamma)\};\\
V_{p,\ell}(Q_T)&=\{v\in L^p(0,T;W^{1,p}(\Omega)):\ 
v|_{\Sigma_T}\in L^{\ell}(\Sigma_T)\},
\end{align*}
for $p,\ell >1$. For $\ell\leq p_*$, 
 with $p_*=p(n-1)/(n-p)$ if $n>p$, and any $p_*>p$ if $n=p$,
observe that $ V_{p,\ell}(Q_T)= L^p(0,T;W^{1,p}(\Omega))$
due to 
 the trace embedding
 $W^{1,p}(\Omega)\hookrightarrow L^{p_*}(\Gamma)$.

Let us introduce the definition of a closed operator 
that admits maximal parabolic regularity on a Banach space $X$ 
\cite{arendt,hall}.
\begin{dfn}
We say that  \emph{$B$ admits maximal parabolic regularity on $X$} if
$B$ is a closed (not necessarily linear)
operator in $X$ with dense domain $D(B)$, and for any
 $F\in L^p(0,T;X)$ ($1<p<\infty$) there exists a unique function
 $u\in L^{p}(0,T;D(B))$, such that  $\partial_t u\in L^{p}(0,T;X)$,
 solving the abstract Cauchy problem
\[ 
\mbox{(ACP)}\qquad\left\{\begin{array}{l}
{d\over dt}u(t)+Bu(t)=F(t),\qquad \mbox{a.e. }t\in ]0,T[\\
u(0)=u_0\in (D(B),X)_{1/p,p}=(X,D(B))_{1/p',p}
\end{array}\right.
\]
where $(D(B),X)_{1/p,p}=\{ v(0):\
v\in L^{p}(0,T;D(B)),\ \partial_t v\in L^{p}(0,T;X)\}$ 
represents  the interpolation space 
\cite[Theorem 5.12]{gris}, and $D(B)$ is endowed with the graph norm.
\end{dfn}
Recall that a densely defined closed operator $B$, such that 
there exists a unique solution of (ACP) for all initial values in $D(B)$,
may be not a generator \cite{aren}. 
For every $u_0\in D(B)$, if $B$ is linear then
this abstract Cauchy problem (ACP)
has the mild solution $u\in C([0,T[;H)$ that verifies
the variation of constants formula
\[
u(t)=\exp[-tB]u_0+\int_0^t \exp[-(t-\tau)B]
F(\tau)\mathrm{d\tau},\qquad t\in [0,T[.
\]
Moreover, the fractional powers $B^{1/2}$ and  $B^{-1/2}$ exist and
global strong solutions can be obtained \cite{rank}. 

Here we consider the nonlinear operator 
$B:V_{2,\ell}(\Omega)\rightarrow [V_{2,\ell}(\Omega)]'$ defined by 
\[
\langle Bu,v\rangle:=
\int_\Omega  (\mathsf{A}\nabla u)\cdot\nabla v\mathrm{dx}
+\int_{\Gamma} b(u)u v\mathrm{ds},\quad\forall v\in V_{2,\ell}(\Omega),
\]
with the assumptions on the coefficients
$\mathsf{A}$ and $b$ being
\begin{itemize}
\item
 $\mathsf{A}=[A_{ij}]_{i,j=1,\cdots,n}$ is a
bounded measurable   $(n\times n)$ matrix-valued function such that
\begin{equation}\label{amin}
\exists  a_\#>0,\quad 
A_{ij}(x)\xi_i\xi_j\geq a_\#|\xi|^2,
\quad\mbox{ a.e. }x\in\Omega,\ \forall \xi\in\mathbb{R}^n,
\end{equation}
under the summation convention over repeated indices:
$\mathsf{A}{\bf a}\cdot{\bf b}=A_{ij}a_jb_i={\bf b}^\top  \mathsf{A}{\bf a}$.
\item $b:\Gamma\times\mathbb R\rightarrow\mathbb R$ is
 a Carath\'eodory function,
 that is, it is measurable in $\Gamma$ and continuous in
  $ \mathbb R$. Moreover,  $b$ is monotone with respect with the
second variable, and it verifies
for some $\ell\geq 2$
\begin{eqnarray}
\label{bmin}\exists b_\#>0:&&
b(\cdot,s) \geq b_\#|s|^{\ell-2};\\
\exists\gamma_1\in
L^\infty(\Gamma):&&|b(x,s)|\leq\gamma_1
(x)|s|^{\ell-2},\label{gama1} 
\end{eqnarray}
for all
 $ s,t\in\mathbb R,$ and  a.e. in $\Gamma$.
\end{itemize}
Set
\begin{equation}
  a^\#=\|\mathsf{A}\|_{\infty,\Omega},
\quad b^\#= \|\gamma_1\|_{\infty,\Gamma}.\label{abmax}
\end{equation}

Under the assumptions (\ref{amin})-(\ref{gama1}), $B$
is monotone, hemicontinuous, bounded, and coercive.
The existence and uniqueness of $u$ of (ACP)
is consequence of \cite[Theorem 4.1, p. 120]{show} if
provided by $u_0\in L^2(\Omega)$.
In particular,
 $B$ admits maximal parabolic regularity on $[V_{2,\ell}(\Omega)]'$,
and its
negative $-B$ generates a $C_0$-semigroup 
on  $[V_{2,\ell}(\Omega)]'$. 

We seek for the $L^p$-integrability of the gradient of $u$
that verifies the variational formulation
\begin{eqnarray}
\int_0^T\langle \partial_t u, v\rangle\mathrm{dt}+\int_{Q_T}
( \mathsf{A}\nabla u)\cdot \nabla v\mathrm{dx}
\mathrm{dt}+\int_{\Sigma_T} b(u)uv\mathrm{ds}\mathrm{dt}=\nonumber\\
=\int_{Q_T} {\bf f}\cdot \nabla v\mathrm{dx}
\mathrm{dt}+
\int_{Q_T}fv \mathrm{dx}\mathrm{dt}+
\int_{\Sigma_T} hv\mathrm{ds}\mathrm{dt}\label{wvf}
\end{eqnarray}
for every $v\in V_{2,\ell}(Q_T)$.
The symbol $\langle\cdot,\cdot\rangle$ stands for the duality pairing in
which is meaningful.

Let $p,\ell\geq 2$.
We denote by $\mathcal{W}_{p,\ell}$ the set of functionals
 $F\in [V_{p',\ell}(Q_T)]'$ that are the form defined by
\[
F(v):=\int_{Q_T}{\bf f}\cdot \nabla v\mathrm{dx}\mathrm{dt}+
\int_{Q_T}fv \mathrm{dx}\mathrm{dt}+
\int_{\Sigma_T}hv\mathrm{ds}\mathrm{dt}
,\quad\forall v\in V_{p',\ell}(Q_T),
\]
with ${\bf f}\in L^p(0,T;{\bf L}^p(\Omega))$, $f\in L^p(Q_T)$,
and 
$h\in L^{\ell/(\ell-1)}(\Sigma_T)$.
The identification $L^p(0,T;[V_{p',\ell}(\Omega)]')
\equiv [L^{p'}(0,T;V_{p',\ell}(\Omega))]'$
 is due to the Phillips
Theorem if provided that $V_{p,\ell}(\Omega)$ is reflexive
and $1<p<\infty$ \cite[p. 104]{show}.
 We simply write $ \mathcal{M}_p= \mathcal{M}_{p,2}$, and
$ V_{p}(Q_T)= V_{p,2}(Q_T)$.
 
We state our main result in the following theorem.
\begin{thm}\label{tmain}
Let $\Omega$ be a $C^{1}$ domain,  $T>0$, and the assumptions
 (\ref{amin})-(\ref{gama1}) be fulfilled.
There exists $\delta>0$ such that for any $p\in [2,2+\delta]$
if  ${\bf f}\in {\bf L}^{2+\delta}(Q_T)$, $f\in L^{2+\delta}(Q_T)$, 
$h\in L^{2+\delta}(\Sigma_T)$ and
$u_0\in L^{2+\delta}(\Omega)$,
then there exists a function $u$ in $L^{p,\infty}(Q_T)\cap V_{p,\ell+p-2}(Q_T)$
which is solution of (\ref{wvf}) such that
 \begin{eqnarray}
{\rm ess} \sup_{t\in [0,T]}\| u\|_{p,\Omega}^p(t) \leq
\mathcal{G}(a_\#,b_\#,p)
\exp\left[(p-2+(p-1)\nu_0^{1/(p-1)})T\right];\label{gr1}\\
\|u\|^{\ell+p-2}_{\ell+p-2,\Sigma_T}\leq 
(b_\#)^{-1}\mathcal{E}(a_\#,b_\#,p);\\
\|\nabla u\|_{p,Q_T}\leq \mathcal{M}(a_\#,b_\#)
,\label{cota2}
\end{eqnarray}
with
\begin{eqnarray}
\mathcal{G}(a_\#,b_\#,p)&=& \| u_0\|_{p,\Omega}^p+\frac{1}{\nu_0}
\|f \|_{p,Q_T}^p+ \left({p-1\over a_\#}\right)^{p/2}
\|{\bf f}\|_{p,Q_T}^p+ \nonumber\\&&
+\frac{p(\ell-1)}{(\ell+p-2)
 b_\#^{(p-1)/(\ell-1)}}\int_{\Sigma_T}
|h|^{\frac{\ell+p-2}{\ell-1}} \mathrm{ds}\mathrm{d\tau};
\nonumber \\
\mathcal{E}(a_\#,b_\#,p)&=&
\mathcal{G}(a_\#,b_\#,p)\left(1+(p-2+(p-1)\nu_0^{1/(p-1)})T\times
\right.\nonumber\\&&
\left. \times\exp\left[(
p-2+(p-1)\nu_0^{1/(p-1)})T\right]\right)
;\label{eaa}\\
\mathcal{M}(a_\#,b_\#)&=&C(n)
\left[\sqrt{\frac{\mathcal{E}(a_\#,b_\#,2)}{a_\#}}+
\frac{\left(1 +\upsilon\right)^{1/p}}{a_\#}
 \left(\sqrt{1+a_\#} \|{\bf f}\|_{p,Q_T}
+\right.\right.\nonumber \\ &&\left. \left.
 +
\frac{1}{\sqrt{\nu_0}}\|f\|_{p,Q_T}+
\sqrt{1+a_\#}K_{{2n/( n+1)}}\|h\|_{p,\Sigma_T} \right)\right]
.\label{maa}
\end{eqnarray}
Here,   $\nu_0=\nu_0(f)$ is a positive constant if   $f\not=0$,
and $\nu_0(0)=0$ otherwise; $C(n)$ is according to (\ref{const}), and
 $K_{2n/(n+1)}$ stands for  the continuity constant of the trace embedding
$W^{1,2n/(n+1)}(\Omega)\hookrightarrow L^2(\Gamma)$.
In particular,  $\partial_t u\in \mathcal{W}_{p,\ell}$.
\end{thm}
If $ b_\#=0$ in (\ref{bmin}),
$B$ is $L^2(\Omega)$-elliptic:
$\langle Bw,w\rangle+ a_\#\|w\|_{2,\Omega}\geq a_\#\|w\|_{1,2,\Omega}
$, but is not coercive on $H^1(\Omega)$.
However, it is possible to reformulate the above theorem such that
similar estimates may be obtained by the 
Gehring-Giaquinta-Modica theory if $\ell\leq 2_*$ is provided, {\em i.e.}
 $ W^{1,2}(\Omega)\hookrightarrow L^{\ell}(\Gamma)$
for any $\ell \geq 2$ if $n=2$, and $\ell\leq 2(n-1)/(n-2)$ if $n>2$
(cf. Remark \ref{rb0}).

\section{$L^{p,\infty}(Q_T)$ and $L^{\ell+p-2}(\Sigma_T)$ estimates}
\label{slpinf}

Local $L^{p,\infty}(Q_T)$-estimates can be obtained under the
Gehring-Giaquinta-Modica technique  as can be found in \cite{ark95}.
Under the Moser technique as already developed
in \cite{lc-arx},
$L^{p,\infty}(Q_T)$ and $L^{\ell+p-2}(\Sigma_T)$ estimates can
appear as consequence of $L^{\infty}(Q_T)$ and $L^{\infty}(\Sigma_T)$
estimates, respectively. Here we provide the explicit estimates
under the direct "apriori" technique
in the following proposition.
\begin{prop}\label{pinf}
Any function  $u$ solving (\ref{wvf}) satisfies, for all $p\geq 2$,
 (\ref{gr1}) and
 \begin{eqnarray}
a_\#\||u|^{(p-2)/2}\nabla u\|_{2,Q_T}^2+
b_\#\int_0^T\int_{\Gamma} |u|^{\ell+p-2}\mathrm{ds}\mathrm{d\tau}
\leq\mathcal{E}(a_\#,b_\#,p),\label{gr2}
\end{eqnarray}
with  $\mathcal{E}(a_\#,b_\#,p)$ being given by (\ref{eaa}).
\end{prop}
\begin{proof}
Fix $t\in ]0,T[$ arbitrary, and let $\chi_{]0,t[}\in L^\infty (]0,T[)$
be the characteristic function.
 Taking $v=\chi_{]0,t[}|u|^{p-2}u$ as a test function in (\ref{wvf}),
 applying the H\"older and Young 
inequalities, and using (\ref{amin}) and (\ref{bmin}),  
we obtain
 \begin{eqnarray*}
 {1\over p}\| u\|_{p,\Omega}^p(t)+
a_\#(p-1)\||u|^{(p-2)/2}\nabla u\|_{2,Q_t}^2+b_\#
\int_0^t\int_{\Gamma} |u|^{\ell+p-2}\mathrm{ds}\mathrm{d\tau}
\leq\\ \leq
 {1\over p}\| u_0\|_{p,\Omega}^p+
\int_0^t\|f \|_{p,\Omega}\|u \|_{p,\Omega}^{p-1}\mathrm{d\tau}+
\\
+ {p-1\over 2a_\#}
 \int_0^t\|{\bf f}\|_{p,\Omega}^2\|u\|_{p,\Omega}^{p-2}\mathrm{d\tau}+
 {a_\#(p-1) \over 2}\||u|^{(p-2)/2}\nabla u\|_{2,Q_t}^2+ \\ +
{\ell -1\over (\ell+p-2) b_\#^{(p-1)/(\ell-1)}}\int_0^t\int_{\Gamma}
| h|^{\frac{\ell+p-2}{\ell-1}} \mathrm{ds}\mathrm{d\tau}+
\frac{b_\#(p-1)}{\ell+p-2}
\int_0^t\int_{\Gamma} |u|^{\ell+p-2}\mathrm{ds}\mathrm{d\tau}.
\end{eqnarray*}
Rearranging the terms, we have
 \begin{eqnarray*}
 {1\over p}\| u\|_{p,\Omega}^p(t)+
\frac{a_\#(p-1)}{2}\||u|^{(p-2)/2}\nabla u\|_{2,Q_t}^2+
\frac{b_\#(\ell-1)}{\ell+p-2}
\int_0^t\int_{\Gamma} |u|^{\ell+p-2}\mathrm{ds}\mathrm{d\tau}
\leq\\ \leq
 \frac{1}{p}\left(\| u_0\|_{p,\Omega}^p+\frac{1}{\nu_0}
\|f \|_{p,Q_t}^p\right)+\frac{1}{p}\left(p-2+(p-1)\nu_0^{1/(p-1)}\right)
 \int_0^t\|u \|_{p,\Omega}^{p}\mathrm{d\tau}+
\\
+ \frac{2}{p}\left({p-1\over 2a_\#}\right)^{p/2}
\|{\bf f}\|_{p,Q_t}^p+
{\ell -1\over (\ell+p-2) b_\#^{(p-1)/(\ell-1)}}\int_0^t\int_{\Gamma}
|h|^{\frac{\ell+p-2}{\ell-1}} \mathrm{ds}\mathrm{d\tau}.
\end{eqnarray*}
By Gronwall inequality, we find (\ref{gr1}), and consequently
(\ref{gr2}) holds.
\end{proof}
\begin{rem}\label{rsp}
If $f\in L^p(0,T;L^{pn/(p+n)}(\Omega))$,
 Proposition \ref{pinf} remains valid with
alternative estimates by considering
\begin{eqnarray*}
\int_\Omega f|u|^{p-2}u\mathrm{dx}\leq
\|f \|_{pn/(p+n),\Omega}S_{p'}\left(\||u|^{p-1} \|_{p',\Omega}+(p-1)
\||u|^{p-2}\nabla u \|_{p',\Omega}
\right)\leq\\
\leq S_{p'}\|f \|_{pn/(p+n),\Omega}\left(\|u \|_{p,\Omega}^{p-1}+(p-1)
\|u \|_{p,\Omega}^{(p-2)/2}\||u|^{(p-2)/2}\nabla u \|_{2,\Omega}
\right)\leq\\
\leq \frac{1+(p-1)^p}{p\nu_0}(S_{p'})^p
\|f \|_{pn/(p+n),\Omega}^p+\frac{\nu_0^{1/p}}{2}
\||u|^{(p-2)/2}\nabla u \|_{2,\Omega}^2+\\+
\left(\frac{(p-1)\nu_0^{1/(p-1)}}{p}+\frac{(p-2)\nu_0^{1/(p-2)}}{2p}
\right) \|u \|_{p,\Omega}^{p},
\end{eqnarray*}
where $S_{p'}$ ($p'<2$) is the continuity constant from
the  embedding  $ W^{1,p'}(\Omega)\hookrightarrow  L^{pn/(pn-n-p)}(\Omega)$.
\end{rem}

\begin{rem}\label{rb0}
The estimates (\ref{gr1}) and (\ref{cota2})  under  $ b_\#=0$ read, respectively,
 \begin{eqnarray*}
{\rm ess} \sup_{t\in [0,T]}\| u\|_{p,\Omega}^p(t) \leq
\mathcal{G}(a_\#,p)
\exp\left[(p-1)(1+\nu_0^{1/(p-1)})T\right];\\
\|\nabla u\|_{p,Q_T}\leq C(n)
\left[\sqrt{\frac{\mathcal{G}(a_\#,2)
\left(1+(1+\nu_0)T\exp\left[(1+\nu_0)T\right]\right)
}{a_\#}}+
\right. \nonumber\\ \left.+
\frac{\left(1 +\upsilon\right)^{1/p}}{a_\#}
 \left(\sqrt{1+a_\#} \left(\|{\bf f}\|_{p,Q_T}+
K_{{2n/( n+1)}}\|h\|_{p,\Sigma_T} \right)
 +
\frac{1}{\sqrt{\nu_0}}\|f\|_{p,Q_T}\right)\right],
\end{eqnarray*}
with
\begin{eqnarray*}
\mathcal{G}(a_\#,p)= \| u_0\|_{p,\Omega}^p+\frac{1}{\nu_0}
\|f \|_{p,Q_T}^p+ \left({p-1\over a_\#}\right)^{p/2}
\|{\bf f}\|_{p,Q_T}^p+ \nonumber\\
+(p-1)\left( \left(\frac{p^2}{2a_\#(p-1)}\right)^{1/(p-1)}
+1\right)K_{2n/(n+1)}^{2/(p-1)}|\Omega|^{[(p-1)n]^{-1}}
 \|h\|_{p',\Sigma_T}^{p'}.
\end{eqnarray*}
 To this aim, it is sufficient to consider in the proof of 
Proposition \ref{pinf} 
\begin{eqnarray*}
\int_{\Gamma} hu\mathrm{ds}\leq \|h\|_{p',\Gamma} \||u|^{p/2}\|_{2,\Gamma}
^{2/p}\leq \\ \leq
 \|h\|_{p',\Gamma} K_{2n/(n+1)}^{2/p}|\Omega|^{(pn)^{-1}}
\left(\frac{p}{2}\||u|^{(p-2)/2}\nabla u\|_{2,\Omega}+\|u\|_{p,\Omega}^{p/2}
\right)^{2/p} \leq \\ \leq
\frac{1}{p}\|u\|_{p,\Omega}^{p}+
\frac{a_\#(p-1)}{2p}\||u|^{(p-2)/2}\nabla u\|_{2,\Omega}^2+ \\ +
 \frac{1}{p'}\left( \left(\frac{p^2}{2a_\#(p-1)}\right)^{1/(p-1)}
+1\right)K_{2n/(n+1)}^{2/(p-1)}|\Omega|^{[(p-1)n]^{-1}}
 \|h\|_{p',\Gamma}^{p'}.
\end{eqnarray*}
\end{rem}

\section{Auxiliary results}
\label{sp}

First, let us state a Caccioppoli-type inequality, under
letting $U\in L^2(\mathbb{R})$ be defined either by $U\equiv 0$, or by
\[
U(t)=\left(
\int_\Omega \eta^2 (x)\mathrm{dx}\right)^{-1}
\int_\Omega \eta^2 (x)u(x,t)\mathrm{dx},\quad \mbox{ if }
{\rm supp}(\eta)\cap\Gamma =\emptyset,
\]
where  $\eta\in W^{1,\infty}_0(\mathbb{R}^n)$ satisfies
 $0\leq \eta\leq 1$ in $\mathbb{R}^n$.
\begin{prop}\label{cacc}
Let $\Omega$ be a $C^{1}$ domain, and $T>0$. 
 If there exists $\delta>0$ such that
  ${\bf f}\in {\bf L}^{2+\delta}(Q_T)$, $f\in L^{2+\delta}(Q_T)$, 
$h\in L^{2+\delta}(\Sigma_T)$ and
$u_0\in L^{2+\delta}(\Omega)$, then 
any function $u \in V_{2,\ell}(Q_T)\cap C([0,T];[V_{2,\ell}(\Omega)]')$
 solving (\ref{wvf}) verifies
\begin{eqnarray}
 {{\rm ess}\hspace*{-0.6cm}}
\sup_{t\in]t_0-R^2,t_0+R^2[}
\|\eta( u-U)\|_{2,\Omega}^2(t)+a_\#(1-(\nu_1+\nu_2))
\|\nabla u\|^{2}_{2,Q_r(z_0)}
\leq \nonumber\\
\leq \left( \frac{2(a^\#)^2}{a_\#}+2+\nu_0+
\frac{3\nu_2}{2}\right)\frac{2}{(R-r)^2}
\|\eta(u-U) \|_{2,Q_R(z_0)}^2
+\frac{1}{\nu_0} \| f\|_{2,Q_R(z_0)}^2 +\nonumber\\
+\left(\frac{1}{a_\#\nu_1}+2\right) \|{\bf f}\|_{2,Q_R(z_0)}^2
+2R\frac{(K_{{2n/( n+1)}})^2 }{\nu_2}
\left(\frac{1}{a_\#}+2\right)
\| h\|_{2,\Sigma_R(z_0)}^{2},
\label{ffh}
\end{eqnarray}
for every $z_0=(x_0,t_0)\in\overline\Omega\times [0,T]$, and $0<R<\sqrt{T}$.
 Here  $\nu_0=\nu_0(f)$,  $\nu_1=\nu_1({\bf f})$, and  $\nu_2=\nu_2(h)$
 are positive constants if   $f\not=0$, ${\bf f}\not={\bf 0}$,
 $h\not=0$, respectively,
and $\nu_0(0)=\nu_1({\bf 0})=\nu_2(0)=0$ otherwise; and $u$
(analogously for each function $f$, $\bf f$, $h$, and $U$) should be
understood as
\begin{equation}\label{tild}
\widetilde u(x,t)=\left\{
\begin{array}{lc}
u(x,-t),\quad &-T<t\leq 0 \\
u(x,t),& 0<t<T\\
u(x,2T-t),& T\leq t<2T
\end{array}\right.
.\end{equation}
\end{prop}
\begin{proof}
Fix  $-T<t_\#<t_1<t_2<t^\#<2T$, and $t\in ]0,T[\cap ]t_\#,t^\#[$.
Let  $\eta\in W^{1,\infty}_0(\mathbb{R}^n)$ satisfy
 $0\leq \eta\leq 1$ in $\mathbb{R}^n$,
and $\zeta\in W^{1,\infty}(\mathbb{R})$ be
  such that  $0\leq \zeta\leq 1$ in $\mathbb{R}$,
$\zeta\equiv 1$ in $]t_1,t_2[$
and $\zeta\equiv 0$ in $\mathbb{R}\setminus [t_\#,t]$.
Since $ W^{1,\infty}_0(\mathbb{R})\hookrightarrow C(\mathbb{R})$
then $\zeta(t_\#)=0$.
If $ t_\#,t_1<0$ 
and/or  $t_2, t^\#>T$, since  $u$ (analogously $f$, $\bf f$, $h$, and $U$)
is only defined on $]0,T[$, then the  extension (\ref{tild}) should be
taken into account.
For the sake of simplicity, we write briefly $u$ instead of $\widetilde u$
(analogously for each function $f$, $\bf f$, $h$, and $U$).

 Taking $v(x,\tau)=\chi_{]-T,t[}(\tau)\zeta^2(\tau)
\eta^2(x)(u(x,\tau)-U(\tau))\in V_{2,\ell}
(Q_T)$ as a test function in (\ref{wvf}),
 making use of (\ref{amin})-(\ref{bmin}) and (\ref{abmax}),  
standard computations yield
  \begin{eqnarray}
\frac{1}{2}\zeta^2\|\eta(u-U)\|_{2,\Omega}^2\Big|^{t}_{t_\#}+
 \frac{a_\#}{2}
\|\zeta\eta \nabla u\|_{2,\Omega\times ]t_\#,t[}^2+b_\#
\int^{t}_{t_\#}\|\zeta\eta (u-U)\|_{\ell,\Gamma}^\ell\mathrm{dt}
\leq\nonumber \\ \leq
\int^{t}_{t_\#}\zeta|\zeta'|\|\eta(u-U)\|_{2,\Omega}^2\mathrm{dt}+
\left( \frac{2(a^\#)^2}{a_\#}+1\right)
\|(u-U)\nabla \eta \|_{2,\Omega\times]t_\#,t[}^2+\nonumber\\+
 \frac{a_\#\nu_1}{2}
\|\zeta\eta \nabla u\|_{2,\Omega\times ]t_\#,t[}^2+
\left(\frac{1}{2\nu_1a_\#}+1\right) \|\eta{\bf f}\|_{2,\Omega\times]t_\#,t[}^2
+\nonumber\\
 +
\int^{t}_{t_\#}\|\zeta\eta f\|_{2,\Omega}\|\zeta\eta(u-U)
 \|_{2,\Omega}\mathrm{dt}
 +
\int^{t}_{t_\#}\|\zeta\eta h\|_{2,\Gamma}
\|\zeta\eta(u-U) \|_{2,\Gamma}\mathrm{dt}. \label{fgh}
\end{eqnarray}

Making use of the trace constant $K_{2n/(n+1)}$ correspondent to the function
$\eta (u-U)\in W^{1,2n/(n+1)}(\Omega)$ with $2n/(n+1)<2\leq n$, 
and after applying the Young inequality, the last boundary integral
in (\ref{fgh}), denoted by $I$, can be computed as
\begin{eqnarray*}
I
\leq 
\int^{t}_{t_\#}\|\zeta\eta h\|_{2,\Gamma}
K_{{2n/( n+1)}}|{\rm supp}(\eta)|^{{1\over 2n}} \zeta \left(
\|\eta\nabla u\|_{2,\Omega} +\right.\\
\left.+ 
\|(u-U)\nabla\eta\|_{2,\Omega }+ \|(u-U)\eta\|_{2,\Omega}\right)
\mathrm{dt}\leq \\
\leq(
K_{{2n/( n+1)}})^2 
|{\rm supp}(\eta)|^{{1\over n}}
\left(\frac{1}{2a_\#\nu_2}+\frac{1}{\nu_2}\right)
\|\eta h\|_{2,\Gamma\times]t_\#,t[}^2 +\\+
 \frac{a_\#\nu_2}{2}
\|\zeta\eta \nabla u\|_{2,\Omega\times ]t_\#,t[}^2
+
\frac{\nu_2}{2}\int^{t}_{t_\#}\zeta^2\left(
\|(u-U)\nabla\eta\|_{2,\Omega }^2+ \|(u-U)\eta\|_{2,\Omega}^2\right)
\mathrm{dt}
.
\end{eqnarray*}
Applying the Young inequality in (\ref{fgh}), and inserting the above
inequality, we deduce
  \begin{eqnarray*}
\frac{1}{2}\ {{\rm ess}\hspace*{-0.1cm}}
\sup_{t\in]t_\#,t^\#[}\zeta^2\|\eta(u-U)\|_{2,\Omega}^2(t)+
 \frac{a_\#}{2}\left(1-(\nu_1+\nu_2)\right)
\|\eta \nabla u\|_{2,\Omega\times ]t_1,t_2[}^2
\leq\nonumber \\ \leq
\int^{t^\#}_{t_\#}\left(|\zeta'|+\frac{\nu_0}{2}+\frac{\nu_2}{2}\right)
\|\eta(u-U)\|_{2,\Omega}^2\mathrm{dt}
+\frac{1}{2\nu_0}
\|\eta f\|_{2,\Omega\times ]t_\#,t^\#[}^2
 +\nonumber \\+
\left( \frac{2(a^\#)^2}{a_\#}+1+\frac{\nu_2}{2}\right)
\|(u-U)\nabla \eta \|_{2,\Omega\times]t_\#,t^\#[}^2+
\left(\frac{1}{2a_\#\nu_1}+1\right) \|\eta{\bf f}\|_{2,\Omega\times]t_\#,t^\#[}^2
 +\nonumber\\
+(K_{{2n/( n+1)}})^2 
|{\rm supp}(\eta)|^{{1\over n}}
\left(\frac{1}{2a_\#\nu_2}+\frac{1}{\nu_2}\right)
\|\eta h\|_{2,\Gamma\times]t_\#,t^\#[}^{2}.
\end{eqnarray*}

 Then, we conclude (\ref{ffh}),
by taking  $\eta\equiv 1$ in $Q_r(x_0)$, $\eta\equiv 0$ in $\mathbb{R}^n
\setminus Q_R(x_0)$, and $|\nabla\eta|\leq (R-r)^{-1}$
a.e. in $Q_R(x_0)\setminus Q_r(x_0)$
 for any $0<r<R$ such that $(R-r)^2\leq 2$; and $|\zeta'|\leq (R-r)^{-2}$
 with $t_\#=t_0-R^2$, $t_1=t_0-r^2$, $t_2=t_0+r^2$, and $t^\#=t_0+R^2$.
\end{proof}
 
Let us recall a result on the Stieltjes integral in the form that
we are going to use (for the general form see \cite{ark99}).
\begin{lemma}\label{hh}
Suppose that $q\in ]0,\infty[$, and $a\in ]1,\infty[$.
If $h,H_i:[1,\infty[\rightarrow [0,\infty[$ are nonincreasing functions such that
\begin{equation}\label{lim0}
\lim_{\iota\rightarrow\infty}h(\iota)=
\lim_{\iota\rightarrow\infty}H_i(\iota)=0,\qquad i=1,\cdots, M_0,
\end{equation}
and that
\begin{equation}\label{hhq}
-\int_\iota^\infty \tau^q \mathrm{dh(\tau)}\leq a[\iota^qh(t)+
\sum_{i=1}^{M_0}H_i^{\beta_i}(\iota)],\quad
\forall \iota\geq 1,
\end{equation}
with $\beta_i\geq 1$,
then, for $\gamma\in [q,aq/(a-1)[$
\begin{eqnarray}\label{hhtese}
-\int_{1}^\infty \iota^\gamma \mathrm{dh(\iota)}\leq {q\over aq-(a-1)\gamma}
\left(-\int_{1}^\infty \iota^q \mathrm{dh(\iota)}\right) +\nonumber\\
+{a\gamma\over aq-(a-1)\gamma}\sum_{i=1}^{M_0}H_i^{\beta_i-1}(1)
\left(-\int_{1}^\infty \iota^{\gamma-q} \mathrm{dH_i(\iota)}\right).
\end{eqnarray}
\end{lemma}

Throughout this section, $z=(x,t)$ stands for spatiotemporal points.
Under  the parabolic metric in $\mathbb{R}^{n+1}$ being given by
\[
d(z^{(1)},z^{(2})=\max_{i=1,\cdots,n}\{|x^{(1)}_i-x^{(2)}_i|,|
t^{(1)}-t^{(2)}|^{1/2}\},
\] 
we use the following standard notation  for
the parabolic parallelepiped
\begin{eqnarray}
Q_R(z)&:=&\{(y,\tau)\in \mathbb{R}^{n+1}:\
d((y,\tau),(x,t))<R\}\nonumber \\
&=&Q_R^{(n)}(x)\times ]t-R^2,t+R^2[,\label{cube}
\end{eqnarray}
where the  spatial cubic interval $Q^{(n)}_R(x)$ stands for the
 cube with edges parallel to coordinate planes
centered at the point $x$ with the radius $R>0$.
When no confusion arises, we shall omit the space dimension and write briefly
$Q_R(x)$.
Furthermore,
we set
\begin{eqnarray*}
Q^+_{R}(z)&:= &\{(y,\tau)\in Q_R(z): \ y_n>x_n\};\\
\Sigma_{R}(z)&:=&\{y\in Q_R(z): \ y_n=x_n\}.
\end{eqnarray*}

Next, we determine an explicit constant involved in
  the reverse H\"older inequality with increasing supports and
an  additional surface integral, where the data exponents improve the ones in 
 \cite{ark99}.
 Observe that in \cite{ark99} the assumed restriction
 $(n-1)/l_1+2/l_2 \geq  (n+2)/s$ is not true for $l_1=l_2=s$. 
 The elliptic version of the below result is stated in \cite{arxC}.
\begin{prop}\label{surf}
Let $R_0>0$, and $z_0=(x_0,t_0)$ with $x_0=(x_0',0)\in\mathbb {R}^n$
and $t_0\geq 0$.
For $p>1$ and $\delta>0$, 
we suppose that the nonnegative functions
$\Phi\in L^p(Q^+_{R_0}(z_0))$, 
$F\in L^{m_1+\delta,m_2+\delta}(Q^+_{R_0}(z_0))$, 
 $G\in L^{l_1+\delta,l_2+\delta}(\Sigma_{R_0}(z_0))$,
 and $\varphi\in L^{1+\delta}(Q^+_{R_0}(x_0) )$
satisfy the estimate
\begin{eqnarray}
\frac 1{R^{n+2}}\int_{Q_{\alpha R}(z)\cap Q^+_{ R_0}(z_0)} \Phi^p\mathrm{dz}
\leq B
\Big(\frac 1{R^{n+2}}\int_{Q_R(z)\cap Q^+_{ R_0}(z_0)} \Phi \mathrm{dz}\Big)^{p}
+\nonumber
\\ +\frac 1{R^{n+2}}\left(\| F\|_{m_1,m_2,Q_R(z)\cap Q^+_{R_0}(z_0)}^{r}
+\| \varphi\|_{1,Q_R(x)\cap Q^+_{R_0}(x_0)}^d\right)+\nonumber\\
+\frac 1{R^{n+1}}\| G\|_{l_1,l_2,Q_R(z)\cap\Sigma_{R_0}(z_0)}^s
, \label{hisums}
\end{eqnarray}
for all $z\in Q_{R_0}(z_0)$,
 and all $R>0$ such that
$Q_R(z)\cap\Sigma_{R_0}(z_0))\not=\emptyset$ 
and $Q_R(z)\subset\subset Q_{R_0}(z_0)$,
with some constants $\alpha\in [1/2,1[$, $B>0$, and
\begin{eqnarray}
\frac{n}{m_1}+\frac{2}{m_2}\geq \frac{n+2}{r},&&
r\geq m_2\geq m_1\geq 1; \label{m12}\\
\frac{n-1}{l_1}+\frac{2}{l_2}\geq \frac{n+1}{s},&&
s\geq l_2\geq l_1\geq 1; \label{l12}\\
nd\geq n+2.&& \label{d12}
\end{eqnarray}
 Then, $\Phi\in  L^{p+\varepsilon}(\omega\cap Q_{R_0}^+(z_0))$,
for all $\varepsilon\in [0,\delta]\cap [0,(p-1)/(\upsilon -1)[$ and
measurable set $\omega\subset \subset Q_{R_0}(z_0)$.
In particular, if $R_0\leq 3/2$, and
${\rm dist}(\omega,\partial Q_{R_0}(z_0))=\beta R_0$ with $\beta\in ]0, 1[$,
 it verifies
\begin{eqnarray}
\| \Phi\|_{p+\varepsilon,\omega\cap Q_{R_0}^+(z_0)}^{p+\varepsilon}
\leq \frac{\beta^{-(n+2) (1+\varepsilon/p)}}
{ p-1-(\upsilon -1)\varepsilon}\left[\frac{p-1}{R_0^{(n+2)\varepsilon/p}}
\left( \| \Phi\|_{p, Q_{R_0}^+(z_0)}^{p+\varepsilon}+
\right.\right. \nonumber \\ \left.
+ \| \Phi\|_{p, Q_{R_0}^+(z_0)}^p
\left(\| F\|_{m_1,m_2,Q_{R_0}^+(z_0)}^{r
\varepsilon /p}+\frac{2R_0}{3}
\|G\|_{l_1,l_2,\Sigma _{R_0}(z_0)}^{s\varepsilon /p} 
+\frac{4R_0^2}{9}\|\varphi\|_{1,Q_{R_0}^+(x_0)}^{d\varepsilon /p}
\right)\right)\nonumber\\
+\upsilon (p-1+\varepsilon)\left(
E_1 \| F\|_{m_1+r\varepsilon/p,m_2+r\varepsilon/p, Q_{R_0}^+(z_0)}^{
m_2+r\varepsilon /p}
\|F\|_{m_1,m_2,Q_{R_0}^+(z_0)}^{r-m_2}+\right.\nonumber \\ +E_2
\| G\|_{l_1+s\varepsilon /p,l_2+s\varepsilon /p,\Sigma_{R_0}(z_0)}^{l_2+
s\varepsilon /p}
\| G\|_{l_1,l_2,\Sigma_{R_0}(z_0)}^{s-l_2}+
\nonumber\\
\left.\left.+3^{nd-(n+2)}(2R_0)^{n+2-nd}
\| \varphi\|_{1+d\varepsilon /p,Q_{R_0}^+(x_0)}^{1+
d\varepsilon /p}
\| \varphi\|_{1,Q_{R_0}^+(x_0)}^{d-1}
\right)\right],
\qquad\label{highup}
\end{eqnarray}
where $E_1$ and $E_2$ are given by (\ref{e1})-(\ref{e2}), respectively, and
\begin{eqnarray}
\upsilon= (4^{n}+1)\left(
 2^{n+2}( 2^{n+2}B)^{1/p}+
\frac{3}{p'}+2^{\left(\frac{n}{m_1}+\frac{1}{m_2}\right){r\over p}}+
2^{\left(\frac{n-1}{l_1}+\frac{1}{l_2}\right){s\over p}}+ 2^{nd/p}+\right.
\nonumber\\ \left.
+2^{2n+3}\right)^p.\label{defau}
\end{eqnarray}
\end{prop}
\begin{proof}
 We prolong $\Phi$ (analogously $F$)
 and $\Psi$  as even functions with respect to 
 $\Sigma_{R_0}(z_0)$:
\[
\widetilde\Phi(x',x_n)=\Big\{
\begin{array}c
\Phi(x',x_n),\quad x_n>0 \\
\Phi(x',-x_n),\quad x_n<0
\end{array}
\qquad
\widetilde\varphi(x',x_n)=\Big\{
\begin{array}c
\varphi(x',x_n),\quad x_n>0 \\
\varphi(x',-x_n),\quad x_n<0.
\end{array}
\]

Transforming $ Q_{R_0}(z_0)$
 into $Q=Q_{3/2}(0)\times ]-9/4,9/4[$ by the passage to
new coordinates system $(y,\tau)=(3(x-x_0)/(2R_0),9(t-t_0)/(4R_0^2))$,
and setting
\begin{eqnarray*}
M=\frac{3^{n+2}}{(2R_0)^{n+2}}\left(\|\widetilde\Phi\|_{p,Q_{R_0}(z_0)}^p
+\|\widetilde F\|_{m_1,m_2,Q_{R_0}(z_0)}^r+\frac{2R_0}3
\|G\|_{l_1,l_2,\Sigma _{R_0}(z_0)}^s+\right. \\ \left.
+\frac{4R_0^2}{9}\|\widetilde \varphi\|_{1,Q_{R_0}(x_0)}^d\right),
\end{eqnarray*}
 we define  
$\overline\Phi(y,\tau)=M^{-1/p}\widetilde\Phi(x_0+2R_0y/3,t_0+4R_0^2t/9)$,
 $\overline F(y,\tau)=M^{-1/r}\widetilde F(x_0+2R_0y/3,t_0+4R_0^2t/9)$, 
$\overline G(y,\tau)=M^{-1/s}G(x_0+2R_0y/3,t_0+4R_0^2t/9)$, and 
$\overline\varphi(y)=M^{-1/d}\widetilde \varphi(x_0+2R_0y/3)$.

Setting  $\Sigma=\Sigma_{3/2}(0)\times ]-9/4,9/4[$ with
$\Sigma_{3/2}(0)=Q_{3/2}^{(n-1)}(0)\times \{0\}$.
we have
 \[
\max\{\|\overline\Phi\|_{p,Q}^p,\|\overline F\|_{m_1,m_2,Q}^r,
\|\overline G\|_{l_1,l_2,\Sigma}^s,\|\overline \varphi\|_{1,Q_{3/2}^{(n)}(0)}^d
\}\leq 1.
\]
Let us define  $\Phi_0(y,\tau)= 
\overline\Phi(y,\tau)[{\rm dist}((y,\tau),\partial Q)]^{(n+2)/p}$.

In order to apply Lemma \ref{hh}, our objective is to  prove that
\begin{equation}\label{phipsi}
\int_{Q[\Phi_0 >\iota]}\Phi_0^p\mathrm{dy}\mathrm{d\tau}
\leq \upsilon \left(\iota^{p-1}
h(\iota)+H_1^{r/m_2}(\iota)+H_2^{s/l_2}(\iota)+H_3^{d}(\iota)\right),
\end{equation}
for any $\iota\in [1,\infty[$, with  $\upsilon $
being as in (\ref{defau}), and
\begin{eqnarray*}
h(\iota)&=&\int_{Q[\Phi_0 >\iota]}\Phi_0\mathrm{dy}\mathrm{d\tau};\\
H_1(\iota)&=&\int_{-9/4}^{9/4}\left(
\int_{ Q_{3/2}(0)
[ \overline F(\cdot,\tau)>\iota^{p/r}]}\overline F^{m_1}\mathrm{dy}\right)
^{m_2/m_1}\mathrm{d\tau};\\
H_2(\iota)&=&\int_{-9/4}^{9/4}\left(
\int_{ \Sigma_{3/2}(0)[ \overline G
(\cdot,\tau)>\iota^{p/s}]}\overline G^{l_1}\mathrm{ds_y}\right)
^{l_2/l_1}\mathrm{d\tau};\\
H_3(\iota)&=&
\int_{ Q_{3/2}(0)[\overline \varphi>\iota^{p/d}]}\overline 
\varphi\mathrm{dy}.
\end{eqnarray*}

Fix $\iota \geq 1$.
Decomposing $Q=\cup_{k\in\mathbb{N}_0}C^{(k)}=\cup_{k\in\mathbb{N}_0}
\cup_{i=1,\cdots, I}D^{(k)}_i$,
with  $C^{(0)}=Q_{1/2}(0)$, and for each $k\geq 1$,
$C^{(k)}=\{(y,\tau)\in Q:\ 2^{-k}<{\rm dist}((y,\tau)
,\partial Q)\leq 2^{-k+1}\}$, and
 $D^{(k)}_i$ are disjoint cubic intervals 
of size $1/2^{k+2}$ such that 
finitely ($I\in\mathbb{N}$)
decompose each set $C^{(k)}$, $k\in \mathbb{N}_0$,
since 
\[
{1\over |D^{(k)}_i|}\int_{D^{(k)}_i}
\Phi_0^p(y,\tau)\mathrm{dy}\mathrm{d\tau}\leq 2^{3(n+2)},
\]
the parabolic version of the Calderon-Zygmund subdivision argument
 implies that (for details see \cite{ark99,arxC})
if there exists $\lambda>2^{3(n+2)}$ 
then there exists a disjoint sequence of
cubic intervals $Q^{(k)}_{j}=Q_{r_j^{(k)}}(y^{(k,j)},\tau^{(k,j)}
)\subset C^{(k)}$
 such that $r^{(k)}_j< 2^{-(k+3)}$, and
\begin{eqnarray}\label{sum}
\int_{Q[\Phi_0 >\iota\sqrt[p]{\lambda}]}\Phi_0^p\mathrm{dy}
\mathrm{d\tau}\leq 2^{n+2} \iota^p\lambda
\sum_{k\geq 0}\sum_{j\geq 1}|Q^{(k)}_j|;\\
\label{iota}
\iota^p\lambda<
{2^{-(k-1)(n+2)}\over (2r^{(k)}_j)
^{n+2}}\int_{Q^{(k)}_j}\overline\Phi^p\mathrm{dy}
\mathrm{d\tau}.
\end{eqnarray}

Next, in order to estimate the  right hand side in (\ref{sum}),
let us prove, for all $k\geq 0$, and $j\geq 1$, there exists
 $R=R_{kj}\in ]r^{(k)}_j,2r_j^{(k)}]$ that verifies
\begin{eqnarray}\label{y0}
\iota 2^{2n+3}R^{n+2}
< \int_{Q_R[\Phi_0>\iota]}\Phi_0\mathrm{dy}\mathrm{d\tau}+\nonumber\\
+\iota^{-p+1}\left( I_1(R,y^{(k,j)},\tau^{(k,j)})+
I_2(R,y^{(k,j)},\tau^{(k,j)})+I_3(R,y^{(k,j)})\right)
,
\end{eqnarray}
with 
the notation $Q_R= Q_{R}(y^{(k,j)},\tau^{(k,j)})$,
for the points $(y^{(k,j)},\tau^{(k,j)})$ such that $Q_R\cap\Sigma 
$ has positive $(n-1)$-Lebesgue measure, and
\begin{eqnarray*}
I_1(R,x,t)&=&\left(\int_{t-R^2}^{t+R^2}\left(
\int_{\{y\in Q_R(x)
:\ \overline F(y,\tau)>\iota^{p/r}\}}\overline F^{m_1}\mathrm{dy}\right)
^{m_2/m_1}\mathrm{d\tau}\right)^{r/m_2};\\
I_2(R,x,t)&=&\left(\int_{t-R^2}^{t+R^2}\left(
\int_{\{y\in \Sigma_R(x):\ \overline G
(y,\tau)>\iota^{p/s}\}}\overline G^{l_1}\mathrm{ds_y}\right)
^{l_2/l_1}\mathrm{d\tau}\right)^{s/l_2};\\
I_3(R,x)&=&\left(
\int_{\{y\in Q_R(x):\ \overline \varphi(y)>\iota^{p/d}\}}\overline 
\varphi\mathrm{dy}\right)^{d}.
\end{eqnarray*}
Since  $R\leq 2r_j^{(k)}<2^{-(k+1)}$, each
$Q_{R}$ only intersects the sets $C^{(k-1)}$, $C^{(k)}$,
and $C^{(k+1)}$. 
We denote by $\mathcal{T}$ that family 
$\{Q_{R}(y^{(k,j)},\tau^{(k,j)})\}_{k\geq 0,\, j\geq 1}$.

Rewriting (\ref{hisums})  in terms of the new coordinates system,
taking $z=(x_0+2R_0y^{(k,j)}/3,t_0+4R_0^2\tau^{(k,j)}/9)$, and dividing the resultant inequality by $M$, we deduce
\begin{eqnarray*}
\frac{1}{R^{n+2}}
\int_{Q_{\alpha R}}\overline\Phi^p\mathrm{dy}\mathrm{d\tau}\leq
B\left({1\over R^{n+2}}\int_{Q_{R}}
\overline\Phi\mathrm{dy}\mathrm{d\tau}
\right)^p+\nonumber
\\
+\frac{1}{R^{n+2}}
\|\overline F\|_{m_1,m_2,Q_R}^{r} 
+\frac{1}{R^{n+1}}
\|\overline G\|_{l_1,l_2,\Sigma_R}^s
+\frac{1}{R^{n}} 
\|\overline \varphi\|_{1,Q_R^{(n)}}^d,
\end{eqnarray*}
where $ Q_{\alpha R}= Q_{\alpha R}(y^{(k,j)},\tau^{(k,j)})\subset C^{(k)}$,
 $R\in ]r^{(k)}_j, 2r^{(k)}_j]$, and
 $\alpha= r^{(k)}_j/ R \in [1/2,1[$, and taking (\ref{m12})-(\ref{d12})
 and $R_0\leq 3/2$ into account.

Inserting the above inequality into  (\ref{iota}), we obtain
\begin{eqnarray}
(\iota R^{n+2})^p\lambda< 2^{n+2} B 
\left(\int_{Q_{R}}\Phi_0\mathrm{dy}\mathrm{d\tau}
\right)^p+R^{(n+2)(p-1)}\|\overline F\|_{m_1,m_2,Q_R}^{r}+\nonumber\\
+R^{(n+2)(p-1)}\|\overline \varphi\|_{1,Q_R^{(n)}}^d
+R^{(n+2)p-(n+1)} \|\overline G\|_{l_1,l_2,\Sigma_R}^s
.\label{ir}
\end{eqnarray}

Each term of the above right hand side is computed as follows
\begin{eqnarray*}
\int_{Q_{R}}\Phi_0\mathrm{dy}\mathrm{d\tau}\leq 
\int_{Q_{R}[\Phi_0>\iota]}\Phi_0\mathrm{dy}\mathrm{d\tau}+\iota(2R)^{n+2};\\
R^{(n+2)(p-1)/ p}\|\overline F\|_{m_1,m_2,Q_R}^{r/p}\leq\\
\leq
 R^{(n+2)(p-1)/ p} \left( I_1^{1/p} (R,y^{(k,j)},\tau^{(k,j)})+
\iota 2^{\left(\frac{n}{m_1}+\frac{1}{m_2}\right){r\over p}}
 R^{\left(\frac{n}{m_1}+\frac{2}{m_2}\right){r\over p}}\right) \\\leq
\frac{\iota^{-(p-1)}}{p}I_1(R,y^{(k,j)},\tau^{(k,j)})+
\iota R^{n+2}\left(\frac{1}{p'}+2^{(n/m_1+1/m_2)r/p}\right);\\
R^{(n+2)(p-1)/ p}\|\overline \varphi\|_{1,Q_R}^{d/p}\leq
\frac{\iota^{-(p-1)}}{p}I_3(R,y^{(k,j)})+
\iota R^{n+2}\left(\frac{1}{p'}+2^{nd/p}\right);\\
R^{(n+2)-(n+1)/ p}\|\overline G\|_{l_1,l_2,\Sigma_R}^{s/p}\leq
R^{(n+2)(p-1)/ p}\|\overline G\|_{l_1,l_2,\Sigma_R}^{s/p}\leq\\
\leq
\frac{\iota^{-(p-1)}}{p}I_2(R,y^{(k,j)},\tau^{(k,j)})+
\iota R^{n+2}\left(\frac{1}{p'}+2^{((n-1)/l_1+1/l_2)s/p}\right).
\end{eqnarray*}
Defining
\[
\lambda=\left( 2^{n+2}( 2^{n+2}B)^{1/p}+
\frac{3}{p'}+2^{\left(\frac{n}{m_1}+\frac{1}{m_2}\right){r\over p}}+
2^{\left(\frac{n-1}{l_1}+\frac{1}{l_2}\right){s\over p}}+ 2^{nd/p}
+2^{2n+3}\right)^p,
\]
 we gather the above inequalities with (\ref{ir}) obtaining (\ref{y0}).

According to the Vitali covering lemma, there exist
$\sigma\in ]3,4[$ and a sequence of disjoint cubic intervals
$\{Q_{R_i}(y^{(i)},\tau^{(i)})\}_{i\geq 1}$ from the collection
$\mathcal{T}$  such that
\[
\cup_{k\geq 0}\cup_{j\geq 1}Q_R(y^{(k,j)},\tau^{(k,j)})\subset
\cup_{ i\geq 1}Q_{\sigma R_i}(y^{(i)},\tau^{(i)})
\subset Q.\]
Hence,\[
\sum_{k\geq 0}\sum_{j\geq 1}|Q^{(k)}_j|\leq
\sum_{k\geq 0}\sum_{j\geq 1}
|Q_R(y^{(k,j)},\tau^{(k,j)})|\leq\sigma^n \sum_{i\geq 1}|Q_{R_i}(y^{(i)}
,\tau^{(i)})|.
\]
Combining the above with (\ref{sum}), and (\ref{y0}), we find
\[ 
\int_{Q[\Phi_0 >\iota\sqrt[p]{\lambda}]}\Phi_0^p\mathrm{dy}\mathrm{d\tau}\leq
 \lambda\sigma^n \Big(\iota^{p-1}h(\iota ) 
+H_1^{r/m_2}(\iota)+H_2^{s/l_2}(\iota)+H_3^{d}(\iota)\Big),
\]
which implies (\ref{phipsi}), by taking $\upsilon \geq\lambda (\sigma^n +1)$.

We have the relations (for details see \cite{ark99},
if $r\geq m_2$, $s\geq l_2$, and $d\geq 1$)
\begin{eqnarray*}
-\int_1^\infty\iota^{\gamma-p+1}dH_1(\iota)
&\leq&{m_2\over m_1}3^{n(m_2/m_1-1)/(m_1+\delta_1)}
\|\overline F\|_{m_1+\delta_1,m_2+\delta_1,Q}^{m_2+\delta_1};\\-\int_1^\infty\iota^{\gamma-p+1}dH_2(\iota)
&\leq&{l_2\over l_1}3^{(n-1)(l_2/l_1-1)/(l_1+\delta_2)}
\|\overline G\|_{l_1+\delta_2,l_2+\delta_2,\Sigma}^{l_2+\delta_2};\\
-\int_1^\infty\iota^{\gamma-p+1}dH_3(\iota)
&\leq&
\|\overline \varphi\|_{1+\delta_3,Q_{3/2}^{(n)}(0)}^{1+\delta_3},
\end{eqnarray*}
with  $\delta_1=r(\gamma-p+1)/p$,  $\delta_2=s(\gamma-p+1)/p$,
and $\delta_3=d(\gamma-p+1)/p$.

Therefore, Lemma \ref{hh} can be applied, concluding that, for 
$\gamma=p+\varepsilon -1$
such that $p\leq \gamma+1<p+{(p-1)/(\upsilon-1)}$, (\ref{hhtese}) 
implies
\begin{eqnarray*}
\int_{Q[\Phi _0>1]}\Phi_0^{p+\varepsilon }\mathrm{dy}\mathrm{d\tau}\leq
{p-1\over \upsilon(p-1)-(\upsilon-1)\gamma} 
\int_{Q[\Phi_0>1]}\Phi_0^p\mathrm{dy}\mathrm{d\tau}+\\
+{\upsilon\gamma\over \upsilon(p-1)-(\upsilon-1)\gamma}\left(
{m_2\over m_1}3^{n(m_2/m_1-1)\over m_1+r\varepsilon /p}
\|\overline F\|_{m_1+r\varepsilon /p,m_2+r\varepsilon /p,Q}^{m_2+
r\varepsilon /p}H_1^{r/m_2-1}(1)+\right. \\ \left.
+{l_2\over l_1}3^{(n-1)(l_2/l_1-1)\over l_1+s\varepsilon /p}
\|\overline G\|_{l_1+s\varepsilon /p,l_2+s\varepsilon /p,\Sigma}^{l_2+
s\varepsilon /p}H_2^{s/l_2-1}(1)
+
\|\overline \varphi\|_{1+d\varepsilon /p,Q_{3/2}^{(n)}(0)}^{1+
d\varepsilon /p}
H_3^{d-1}(1)\right).
\end{eqnarray*}
On the other hand, since
 $\Phi_0^{p+\varepsilon }\leq\Phi_0^{p}$ a.e. in $Q\setminus Q[\Phi_0>1]$,
 we find for any $\omega\subset\subset Q=Q_{3/2}(0)$
\begin{eqnarray*}
\left[{\rm dist}(\omega,\partial Q)\right]^{(n+2)(1+\varepsilon /p)}
\int_{\omega}\overline\Phi^{p+\varepsilon }\mathrm{dy}\mathrm{d\tau}\leq
{(p-1)3^{n +2}\over (\gamma-\upsilon\varepsilon)2^{n+2}} 
\int_{Q}\overline\Phi^p\mathrm{dy}\mathrm{d\tau}+\\
+{\upsilon\gamma\over \gamma-\upsilon\varepsilon}\left(
{m_2\over m_1}3^{n(m_2-m_1)\over m_1(m_1+r\varepsilon /p)}
\|\overline F\|_{m_1+r\varepsilon /p,m_2+r\varepsilon /p,Q}^{m_2+
r\varepsilon /p}
\|\overline F\|_{m_1,m_2,Q}^{r-m_2}+\right. \\ \left.
+{l_2\over l_1}3^{(n-1)(l_2-l_1)\over l_1( l_1+s\varepsilon /p)}
\|\overline G\|_{l_1+s\varepsilon /p,l_2+s\varepsilon /p,\Sigma}^{l_2+
s\varepsilon /p}
\|\overline G\|_{l_1,l_2,\Sigma}^{s-l_2}
+
\|\overline \varphi\|_{1+d\varepsilon /p,Q_{3/2}^{(n)}(0)}^{1+
d\varepsilon /p}
\|\overline \varphi\|_{1,Q_{3/2}^{(n)}(0)}^{d-1}\right).
\end{eqnarray*}
 Keeping the same designation to the transformed
set $\omega\subset \subset  Q_{R_0}(z_0)$, 
we deduce
\begin{eqnarray*}
\left[{\rm dist}(\omega,\partial Q_{R_0}(z_0))
\frac{3}{2R_0}\right]^{(n+2)(1+\varepsilon /p)}
\int_{\omega}\widetilde\Phi^{p+\varepsilon }\mathrm{dz}\leq
\\ \leq
{(p-1)3^{n +2}\over (p-1-(\upsilon-1)\varepsilon)2^{n+2}} 
M^{\varepsilon/p}\int_{Q_{R_0}(z_0)}\widetilde\Phi^p\mathrm{dz}+\\
+\frac{\upsilon (p-1+\varepsilon)}{ p-1-(\upsilon -1)\varepsilon}
\left(
E_1\|\widetilde F\|_{m_1+r\varepsilon /p,m_2+r\varepsilon /p,Q_{R_0}(z_0)}
^{m_2+r\varepsilon /p}
\|\widetilde F\|_{m_1,m_2,Q_{R_0}(z_0)}^{r-m_2}+\right. \\
+E_2
\| G\|_{l_1+s\varepsilon /p,l_2+s\varepsilon /p,\Sigma_{R_0}(z_0)}^{l_2+
s\varepsilon /p}
\| G\|_{l_1,l_2,\Sigma_{R_0}(z_0)}^{s-l_2}+\\ \left.+
\|\widetilde\varphi\|_{1+d\varepsilon /p,Q_{R_0}(x_0)}^{1+
d\varepsilon /p}
\|\widetilde\varphi\|_{1,Q_{R_0}(x_0)}^{d-1}
3^{nd-(n+2)}(2R_0)^{n+2-nd}\right).
\end{eqnarray*}
with
\begin{eqnarray}
E_1= 
{m_2\over m_1}3^{n(m_2-m_1)\over m_1(m_1+r\varepsilon /p)}
\left(\frac{3}{2R_0}\right)^{n
\left(\frac{m_2+r\varepsilon/p}{m_1+r\varepsilon/p}-
\frac{m_2}{m_1}\right)}\frac{(2R_0)^{
n+2}3^{r(n/m_1+2/m_2)}}{3^{n+2}(2R_0)^{r(n/m_1+2/m_2)}}
; \label{e1} \\
E_2=
{l_2\over l_1}3^{(n-1)(l_2-l_1)\over l_1(l_1+s\varepsilon /p)}
\left(\frac{3}{2R_0}\right)^{(n-1)
\left(\frac{l_2+s\varepsilon/p}{l_1+s\varepsilon/p}-\frac{l_2}{l_1}\right)}
\frac{(2R_0)^{n+2}3^{s((n-1)/l_1+2/l_2)}}{3^{n+2}(2R_0)^{s((n-1)/l_1+2/l_2)}}
.\label{e2}
\end{eqnarray}
Therefore, by applying (\ref{m12})-(\ref{d12})
we conclude (\ref{highup}) which completes the proof.
\end{proof}

In a similar manner that we have Proposition \ref{surf} the following 
Proposition can be obtained.
\begin{prop}\label{ppint}
Under the conditions of Proposition \ref{surf},
if instead of (\ref{hisums}),
\begin{eqnarray}
\frac 1{R^{n+2}}\int_{Q_{\alpha R}(z)} \Phi^p\mathrm{dz}
\leq B_{\rm I}
\left(\frac 1{R^{n+2}}\int_{Q_R(z)} \Phi \mathrm{dz}\right)^{p}
+\frac 1{R^{n+2}}\| F\|_{m_1,m_2,Q_R(z)}^{r}+\nonumber
\\
+\frac 1{R^{n+2}}
\| \varphi\|_{1,Q_R(x)}^d, \label{hiint}
\end{eqnarray}
holds for $R<\min\{\sqrt{T},{\rm dist}(x,\partial\Omega)/\sqrt{n}\}$,
 then $\Phi\in  L^{p+\varepsilon}(Q_{r}(z_0))$,
for all $\varepsilon\in [0,\delta]\cap [0,(p-1)/(\upsilon_{\rm I} -1)[$, with
\begin{equation}\label{defai}
\upsilon_{\rm I}= (4^{n}+1)\left(
 2^{n+2}( 2^{n+2}B_{\rm I})^{1/p}+
\frac{3}{p'}+2^{\left(\frac{n}{m_1}+\frac{1}{m_2}\right){r\over p}}+
 2^{nd/p}
+2^{2n+3}\right)^p,
\end{equation}
 and $r=(1-\beta )R_0$
with $\beta\in ]0,1[$. In particular,  it verifies
\begin{eqnarray}
\| \Phi\|_{p+\varepsilon,Q_r(z_0)}^{p+\varepsilon}
\leq \frac{\beta^{-(n+2)\varepsilon/p}}{p-1-(\upsilon_{\rm I} -1)\varepsilon}
\left[
\frac{p-1}{R_0^{(n+2)\varepsilon/p}}\left(
\| \Phi\|_{p, Q_{R_0}(z_0)}^{p+\varepsilon}+\right.\right.\nonumber\\
+\left. \| \Phi\|_{p, Q_{R_0}(z_0)}^p
\left(\| F\|_{m_1,m_2,Q_{R_0}(z_0)}^{r
\varepsilon /p}
+\frac{4R_0^2}{9}\|\varphi\|_{1,Q_{R_0}(x_0)}^{d\varepsilon /p}
\right)\right)+\nonumber\\
+
\upsilon_{\rm I} (p-1+\varepsilon)\left(
E_1 \| F\|_{m_1+r\varepsilon/p,m_2+r\varepsilon/p, Q_{R_0}(z_0)}^{
m_2+r\varepsilon /p}
\|F\|_{m_1,m_2,Q_{R_0}(z_0)}^{r-m_2}+\right.\nonumber \\ 
\left.\left.+3^{nd-(n+2)}(2R_0)^{n+2-nd}
\| \varphi\|_{1+d\varepsilon /p,Q_{R_0}(x_0)}^{1+d\varepsilon /p}
\| \varphi\|_{1,Q_{R_0}(x_0)}^{d-1}
\right)\right],
\qquad\label{highin}
\end{eqnarray}
where $E_1$ and $E_2$ are given by (\ref{e1})-(\ref{e2}), respectively. 
\end{prop}

\begin{rem}\label{ruu}
If $\varphi=0$, then (\ref{defau})-(\ref{defai}) read
\begin{eqnarray}
\upsilon= (4^{n}+1)\left(
 2^{n+2}( 2^{n+2}B)^{\frac{1}{p}}+
\frac{3}{p'}+2^{\left(\frac{n}{m_1}+\frac{1}{m_2}\right){r\over p}}+
2^{\left(\frac{n-1}{l_1}+\frac{1}{l_2}\right){s\over p}}
+2^{2n+3}\right)^p;\label{defuup}\\
\upsilon_{\rm I}= (4^{n}+1)\left(
 2^{n+2}( 2^{n+2}B_{\rm I})^{1/p}+
\frac{3}{p'}+2^{\left(\frac{n}{m_1}+\frac{1}{m_2}\right){r\over p}}
+2^{2n+3}\right)^p.\label{defuin}
\end{eqnarray}
\end{rem}

Finally, we recall a local Poincar\'e inequality, which proof can be found in
\cite{arxC}.
\begin{lemma}\label{lpoin}
For any $x\in\mathbb{R}^n$ and $0<R<\epsilon (2S_{2n/(n+2)})^{-1}$, every 
$u\in H^{1}(Q_R(x))$ verifies
\begin{equation}\label{poiny}
\| u\|_{2,Q_R(x)}\leq {S_{2n/(n+2)}\over
1-\epsilon}\|\nabla u\|_{{2n}/({n+2}),Q_R(x)},
\end{equation}
where $S_{2n/(n+2)}=\pi^{-1/2}n^{(2-3n)/(2n)}(n-2)^{(n-2)/(2n)}[\Gamma(n)/
\Gamma(n/2)]^{1/n}$.
\end{lemma}

\section{Proof of Theorem \ref{tmain}}
\label{sns1p}
 
Let $u \in V_{2,\ell}(Q_T)\cap C([0,T];[V_{2,\ell}(\Omega)]')$ 
 solve (\ref{wvf}) for all $v\in  V_{2,\ell}(Q_T)$. 
 Let  $0<r<R<\sqrt{T}$, and
 $z_0=(x_0,t_0)\in\overline\Omega\times [0,T]$.
 Proposition \ref{cacc} can be applied.

We split the proof by beginning to show the local interior
and lateral higher integrability of the gradient of $u$.

\subsection{Local interior higher integrability of the gradient}
\label{sint}

If $x_0\in\Omega$, we may take $R<{\rm dist}(x_0,\partial\Omega)/\sqrt{n}$.
Considering $R<\sqrt{T}$, $r=R/2\leq 1$,
 $\nu_1=1/2$, and $\nu_2=0$, (\ref{ffh}) reads
  \begin{eqnarray}
 {{\rm ess}\hspace*{-0.6cm}}
\sup_{t\in ]t_0-R^2,t_0+R^2[}\|\eta(u-U)\|_{2,Q_R(x_0)}^2(t)+
\frac{a_\#}{2}
\| \nabla u\|_{2,Q_r(z_0)}^2
\leq\nonumber \\ \leq
\left( \frac{2(a^\#)^2}{a_\#}+2+\nu_0\right)\frac{2^3}{R^2}
\|\eta (u-U) \|_{2,Q_R(z_0)}^2+\nonumber\\+2
\left(\frac{1}{a_\#}+1\right) \|{\bf f}\|_{2,Q_R(z_0)}^2
 +
\frac{1}{\nu_0}\|f\|_{2,Q_R(z_0)}^2.
\label{ffhint}
\end{eqnarray}
In the presence of Lemma \ref{lpoin} it is sufficient to take $U=0$, and we 
restrict to $R<(4S_{2n/(n+2)})^{-1}$.
Denoting by $Y=2S_{2n/(n+2)}$ the constant in the  inequality
(\ref{poiny}), we integrate
over time to obtain
\begin{eqnarray}
\int_{t_0-R^2}^{t_0+R^2}\|\eta u\|_{2,Q_R(x_0)}\|u\|_{2,Q_R(x_0)}
\mathrm{dt}\leq  Y{{\rm ess}\hspace*{-0.6cm}}
\sup_{t\in ]t_0-R^2,t_0+R^2[}
\|\eta u\|_{2,Q_R(x_0)}\times \nonumber\\
\times (2R^2)^{n-2/( 2n)} \|\nabla u\|_{2n/(n+2),Q_R(z_0)}.\label{uuy}
 \end{eqnarray}
Inserting the above inequality into (\ref{ffhint}),
and after applying the Young inequality, we deduce  
\begin{eqnarray*}
\frac{a_\#}{2} \| \nabla u\|_{2,Q_{R/2}(z_0)}^2
\leq \left( \frac{2(a^\#)^2}{a_\#}+2+\nu_0\right)^2
\frac{2^{5-2/n}
Y^2}{R^{2(n+2)/n}}
\|\nabla u\|_{2n/(n+2),Q_R(z_0)} ^2
+\nonumber \\ + 2\left(\frac{1}{a_\#}+1\right) \|{\bf f}\|_{2,Q_R(z_0)}^2
 +
\frac{1}{\nu_0}\|f\|_{2,Q_R(z_0)}^2
.
\end{eqnarray*}

Employing Proposition \ref{ppint}
with $\Phi=|\nabla u|^{2n/(n+2)}$, $p=(n+2)/n$, $m_1=m_2=r=2$,
and
\begin{eqnarray}\label{defbb}
B_{\rm I} = {2^{2(2-1/n)}\over a_\#}
\left( {2(a^\#)^2\over a_\#}+2+\nu_0
 \right)(4S_{2n/(n+2)})^2;\\
F=\left({4({1/ a_\#}+ 1)
|{\bf f}|^{2}+2|f|^{2}/\nu_0\over a_\#}
\right)^{1/2} \in L^{2+\delta}(Q_R(z_0)),\label{defbf}
\end{eqnarray}
 the interior estimate 
  \begin{eqnarray}
\| \nabla u\|_{2+\varepsilon,Q_{(1-\beta)R}(z_0)}
\leq\left(\frac{2n\beta^{-\varepsilon (n+2)/2}}{4-(n+2)(\upsilon_{\rm I} -1)
\varepsilon }\right)^{1/(2+\varepsilon)}\times 
\nonumber\\ \times 
\left[
\left(\frac{2(4+\varepsilon)}{n(2+\varepsilon)
R^{\varepsilon (n+2)/2}}\right)^{1/(2+\varepsilon)}
\| \nabla u\|_{2,Q_{R}(z_0)}+ \right.\nonumber\\ +
 \left(\frac{2^{1+\varepsilon (n+1)/2}
\varepsilon}{n(2+\varepsilon)}+
\upsilon_{\rm I} \left(\frac{4+\varepsilon(n+2)}{2n}\right)
\right)^{1/(2+\varepsilon)}
\times\nonumber \\ \left. \times
\left(\frac{2\sqrt{(1+a_\#)}}{a_\#}
\|{\bf f}\|_{2+\varepsilon,Q_R(z_0)}+
\sqrt{\frac{2}{a_\#\nu_0}}
\|f\|_{2+\varepsilon,Q_R(z_0)}
\right)\right]
\label{ffhi}
\end{eqnarray}
 holds, 
 for any 
$R<\min\{\sqrt{T},{\rm dist}(x_0,\partial\Omega)/\sqrt{n}, 
(4S_{2n/(n+2)})^{-1}\}$, and for all
 $\varepsilon\in [0,\delta]\cap [0,4/[(n+2)(\upsilon_{\rm I} -1)][$ with
 $\upsilon_{\rm I} $ being defined by (\ref{defuin}), {\em i.e.}
\[
\upsilon_{\rm I} =(4^n+1)\left( 2^{2(n+1)}B_{\rm I}^{n/(n+2)}+
\frac{6}{n+2}+2^{n(n+1)/(n+2)}+2^{2n+3}\right)^{n/(n+2)}.
\]

\begin{rem} 
The constant $B_{\rm I}$ defined in (\ref{defbb}) may be differently given.
For instance, it may depend on  the Poincar\'e constant, denoted by
 $C_{\Omega,p}$,
if we use in (\ref{uuy}) the Minkowski, Sobolev, and Poincar\'e inequalities
to successively  compute 
\begin{eqnarray*}
\|\eta(u-U)\|_{2,Q_R(x_0)} \leq  2
\|u-{-\hspace*{-0.4cm}}\int_{Q_R(x_0)}u\mathrm{dx}\|_{2,Q_R(x_0)}\\
\leq  2S_{2n/(n+2)}\left(\|\nabla u\|_{2n/(n+2),Q_R(x_0)} +
\| u-{-\hspace*{-0.4cm}}\int_{Q_R(x_0)}u\mathrm{dx}\|_{2n/(n+2),Q_R(x_0)} \right)
\\
\leq  2S_{2n/(n+2)}\left(1+C_{Q_R(x_0),2n/(n+2)}
\right)\|\nabla u\|_{2n/(n+2),Q_R(x_0)} ,
\end{eqnarray*}
since $u\in W^{1,2n/(n+2)}(Q_R(x_0))$ with $2n/(n+2)<2\leq n$.
 With this approach, the restriction of $R<(4S_{2n/(n+2)})^{-1}$ can be removed.
\end{rem}

\subsection{Local higher 
integrability up to the spatial boundary of the gradient}
\label{scup}

For reader's convenience, we recall the definition of  $C^1$ domain.
We use the notation $y'=(y_1,\ldots, y_{n-1})\in \mathbb R^{n-1}$.
\begin{dfn}\label{cka}
We say that 
$\Omega$ is a domain of class $C^{1}$ (or simply  $C^1$ domain),
 if
$\Omega$ is an open, bounded, connected, nonempty set 
 of $\mathbb R^n$ and it verifies the following:
\[ 
\exists M\in\mathbb N\quad
\exists  \varrho, \nu>0:\qquad\partial\Omega=\cup^M_{m=1}\Gamma_m,
\] 
with 
\begin{enumerate} 
\item 
$\Gamma_m=O^{-1}_m(\{y=(y',y_{n})\in
Q^{(n-1)}_\varrho(0)\times\mathbb{R}:$ $y_{n}=\varpi_m(y')\}$,
\item
$O^{-1}_m(\{y=(y',y_{n})\in Q^{(n-1)}_\varrho(0)\times\mathbb{R}: $
$\varpi_m(y')<y_{n}<\varpi_m(y')+\nu\}) \subset\Omega,$
\item
$O^{-1}_m(\{y=(y',y_{n})\in Q^{(n-1)}_\varrho(0)\times\mathbb{R}:$
$\varpi_m(y')-\nu<y_{n}<\varpi_m(y')\}) \subset \mathbb R^n\setminus \Omega,$
\end{enumerate} 
where
\[
Q^{(n-1)}_\varrho(0)=\{y'=(y_{1},\cdots,y_{n-1})\in
\mathbb{R}^{n-1}:\ |y_{i}|<\varrho,\ i=1,\cdots,n-1\},
\]
and for each  $m=1,\cdots,M$, $ O_m:\mathbb R^n\rightarrow \mathbb R^n$ 
denotes a local coordinate system:
$$y^{(m)}=O_m(x)= \mathsf{O} x+b,\quad 
 \mathsf{O}^{-1}= \mathsf{O}^T,\ \det \mathsf{O}=1;$$
and $ \varpi_m\in C^{1}(Q^{(n-1)}_\varrho(0))$.
\end{dfn}
 
By Definition \ref{cka}, 
there exist $M \in\mathbb{N}$ and $\varrho,\nu>0$ such that
for any $x_0\in\partial\Omega$ there is  $ m\in\{1,\cdots,M\}$ such that a
local coordinate system $y^{(m)}=O_m(x)$
and a local $C^{1}$-mapping $\varpi_m$ verify 
\begin{equation}\label{gm}
x_0\in \Gamma_m=O_m^{-1}\circ \phi_m^{-1}\left(
Q_\varrho^{(n-1)}(0)\times\{0\}\right),
\end{equation}
where 
$\phi_{m}: Q_\varrho^{(n-1)}(0)\times \mathbb{R}
\rightarrow \mathbb R^n$ of class $C^{1}$ is defined by
\begin{equation}
\phi_{m}(y)=
\left(
\begin{array}c
y' \\
y_n-\varpi_{m}(y')
\end{array}
\right).
\label{phi}
\end{equation}
For each  $ m\in\{1,\cdots,M\}$,
we consider  the change of variables
\begin{equation}\label{chg}
y\in Q_\varrho^{(n-1)}(0)\times ]-\nu,\nu [
\quad\mapsto x=O^{-1}(\phi_m^{-1}(y)).
\end{equation}
Since the Jacobian of  the transformation $O^{-1}_m\circ
\phi^{-1}_m$ is equal to 1, let us denote by the same letter any function
$f=f\circ O^{-1}_m\circ\phi^{-1}_m$.

Fix  $x_0\in\partial\Omega$, and
$ m\in\{1,\cdots,M\}$ such that $x_0\in\Gamma_m$ is
in accordance with (\ref{gm}),
set $y_0=\phi_m\circ O_m(x_0)$, and 
\[
\Sigma_R(y_0)=\{y\in  Q_\varrho^{(n-1)}(0)\times ]-\nu,\nu [:\
 |y'-y_0'|<R, \ y_n=0\}, 
\]
for any $0<R\leq  R_0=\min\{\varrho,\nu,{\rm dist}(y_0',\partial'
 Q_\varrho^{(n-1)}(0))\}$.
 Notice that $y_0=(y_0',0)$.
Inasmuch as $\Gamma_0=O_m^{-1}\circ \phi_m^{-1}\left(
\Sigma_{R_0}(y_0)\right)$, different cases occur, namely
 $\Gamma_0\cap\Gamma\not=\emptyset$
and  $\Gamma_0\cap (\partial\Omega\setminus\overline\Gamma)\not=\emptyset$;
 $\Gamma_0\subset\Gamma$, and 
 $\Gamma_0\subset\partial\Omega\setminus\overline\Gamma$.
 Throughout the sequel, we refer to
 $\|\cdot\|_{2,\Sigma_R(y_0)}$ including cases where the set is empty.

 Reorganizing the terms in (\ref{ffh}) with $\nu_2=1/4$
as in Section \ref{sint}, we have
 \begin{eqnarray*}
\|\nabla u\|_{2,Q_{r}^+(z_0)}^2
\leq { B \over R^{2(n+2)/n}} \|\nabla u\|_{2n/(n+2),Q_R^+(z_0)}^2
+ {4\over  a_\#}\left[
2\left( {1\over a_\#}+ 1\right)\|{\bf f}\|_{2,Q_R^+(z_0)}^2
+\right.\\ \left.
+{1\over \nu_0}\|f\|_{2,Q_R^+(z_0)}^2
+8R \left(\frac{1}{a_\#}+2\right) (K_{{2n/( n+1)}})^2 
\|h\|_{2,\Sigma_R(z_0)}^{2}\right].
\end{eqnarray*}
Here $B$ is defined by (compare to (\ref{defbb}))
\[
B= {2^{5-2/n}\over a_\#}
\left( {2(a^\#)^2\over a_\#}+\frac{19}{8}+\nu_0 \right)(4S_{2n/(n+2)})^2.
 \]
 
  Proposition \ref{surf} with  $\Phi=|\nabla u|^{2n/(n+2)}$, $p=(n+2)/n$,
$l_1=l_2=s=2$, and $F$ being defined by,
instead of (\ref{defbf}),
\[
F=2\left({2({1/ a_\#}+ 1)
|{\bf f}|^{2}+|f|^{2}/\nu_0\over a_\#}
\right)^{1/2} \in L^{2+\delta}(Q_R(z_0)),
\]
and 
\[G=
\left[\frac{32}{a_\#}
\left(\frac{1}{a_\#}+2\right) \right]^{1/2}K_{{2n/( n+1)}}|h|
 \in L^{2+\delta}(\Sigma_R(z_0)),
\]
 and the application of the passage to the initial coordinates system 
  upon choosing the neighborhood  $Q_0=O_m^{-1}\circ \phi_m^{-1}\left(Q_{R_0}(y_0)\right)$ of the subset $\Gamma_0$ of the boundary
$\partial\Omega$,
 imply that
\begin{eqnarray}
\| \nabla u\|_{2+\varepsilon,Q_{(1-\beta)R}(z_0)}
\leq\left(\frac{2n\beta^{-\varepsilon (n+2)/2}}{4-(n+2)(\upsilon-1)
\varepsilon }\right)^{1/(2+\varepsilon)}
\times \nonumber\\ \times 
\left[
\left(\frac{2(4+\varepsilon)}{n(2+\varepsilon)
R^{\varepsilon (n+2)/2}}\right)^{1/(2+\varepsilon)}
\| \nabla u\|_{2,Q_{R}(z_0)}+ \right.\nonumber\\ +
 \left(\frac{2^{1+\varepsilon (n+1)/2}
\varepsilon}{n(2+\varepsilon)}+
\upsilon \left(\frac{4+\varepsilon(n+2)}{2n}\right)
\right)^{1/(2+\varepsilon)}
\times\nonumber \\  \times
\left(\frac{2\sqrt{2(1+a_\#)}}{a_\#}
\|{\bf f}\|_{2+\varepsilon,Q_R(z_0)}+\frac{2}{\sqrt{a_\#\nu_0}}
\|f\|_{2+\varepsilon,Q_R(z_0)}
\right)+ \nonumber \\ 
\left.+\left(\frac{2^{1+\varepsilon n/2}\varepsilon}{n(2+\varepsilon)
R^{\varepsilon /2}}+
\upsilon \left(\frac{4+\varepsilon(n+2)}{2n}\right)
\right)^{1\over 2+\varepsilon}
\frac{4}{a_\#}\sqrt{2(1+a_\#)}
K_{{2n\over n+1}}\|h\|_{2
+\varepsilon,\Sigma_R(z_0)}\right],
\label{ffhup}
\end{eqnarray}
for any 
$R<\min\{\sqrt{T}, R_0, (4S_{2n/(n+2)})^{-1}\}$, and
 for all $\varepsilon\in [0,\delta]\cap [0,4/[(n+2)(\upsilon -1)][$ with
 $\upsilon$ being defined by (\ref{defuup})
that is
\begin{equation}\label{upsi}
\upsilon =(4^n+1)\left( 2^{2(n+1)}B_{\rm I}^{n/(n+2)}+
\frac{6}{n+2}+2^{\frac{n(n+1)}{n+2}}
+2^{\frac{1}{n+2}}+2^{2n+3}\right)^{\frac{n}{n+2}}.
\end{equation}

\subsection{Global higher integrability}
\label{scm}

On the one hand, Section \ref{sint} ensures that
for  each point $z\in \Omega\times [0,T]$
it is associated a sequence of cubic intervals $Q_{r(z)/2}(z)$,
with side lengths $r(z)>0$  tending to zero, such that
(\ref{ffhi}) is verified.
On the other hand, Section \ref{scup} ensures that
for  each point $z\in\partial\Omega\times [0,T]$
it is associated a sequence of cubic intervals $Q_{r(z)/2}(z)$,
with side lengths  $r(z)>0$  tending to zero, such that
(\ref{ffhup}) is verified.

From the mathematical point of view, it is indifferent to continue the proof by considering thoses cubic intervals. 
With in mind the view point of real and numerical
applications we prefer to proceed by analysing separately the spatial  domain.

According to the Besicovitch covering theorem
 \cite[Theorem 1.2]{guz}, there exists a sequence of spatial cubic intervals
$\{Q_{r_m/2}(x^{(m)})\}_{m\geq 1}$ from the above collection of
 cubic intervals
 such that: $\overline{\Omega}\subset \cup_{m\geq 1}Q_{r_m/2}(x^{(m)})$; and
every point of $\mathbb{R}^{n}$ belongs to at most $2^{n}+1$ cubes in 
$\{Q_{r_m/2}(x^{(m)})\}_{m\geq 1}$.
Since $\Omega$
 is bounded, this cover is finite, {\em i.e.}
 its cardinal is an integer number $M$.
 Let us define 
\[
r_\#=\min\{r_m:\ m=1,\cdots,M\}.\]

Indeed,  there exists  $N$ (depending on the dimension of the space)
families of pairwise disjoint cubes such that
(for details see \cite{arxC,guz})
 \[
\{Q_{r_m}(x^{(m)})\}_{m=1,\cdots, M}=\cup_{m=1}^N
\{Q_{r_i}(x^{(i)})\}_{i=1,\cdots, \mathcal{I}(m)\cup \mathcal{J}(m)},
\]
where $\mathcal{I}(m)$ contains the indices with $x^{(i)}\in\Omega$,
while $\mathcal{J}(m)$ contains the indices with $x^{(i)}\in\partial\Omega$.
For each $i\in \mathcal{I}(m)$ (analogously for $i\in \mathcal{J}(m)$)
there exists $d=d_i>0$ such that $dr_i^2/4=T$.
If $d<1$, we take $t^{(i)}=0$ observing that $]0,T[\subset ]0,r_i^2/4[$.
If the integer part $\lfloor d\rfloor$
 is even, {\em i.e.} $\lfloor d\rfloor=2k$, $k\in\mathbb{N}$,
then we may build $k+1$ parabolic interval cubes 
$Q_{r_i}(z^{(i,j)})$ centered at $z^{(i,j)}=(x^{(i)},t^{(j)})$
where $t^{(j)}=2(j-1)r_i^2/4$ for $j=1,\cdots, k+1:=\mathcal{K}(i)$,
observing that $t^{(m+1)}=2mr_i^2/4<T$.
If  $\lfloor d\rfloor$ is odd, {\em i.e.}
 $\lfloor d\rfloor=2k-1$, $k\in\mathbb{N}$,
then we may build $k$ parabolic interval cubes 
$Q_{r_i}(z^{(i,j)})$ centered at $z^{(i,j)}=(x^{(i)},t^{(j)})$
where $t^{(j)}=(2j-1)r_i^2/4$ for $j=1,\cdots, k:=\mathcal{K}(i)$.

Hence, combining (\ref{ffhi}) and (\ref{ffhup}) with
\[
\|\nabla u\|_{p,Q_T}\leq \sum_{m=1}^N\left(\sum_{i\in\mathcal{I}(m)
\atop j=1,\cdots,\mathcal{K}(i)}
\|\nabla u\|_{p,Q_{r_i/2}(z^{(i,j)})}+\sum_{i\in\mathcal{J}(m)
\atop j=1,\cdots,\mathcal{K}(i)}
\|\nabla u\|_{p,Q_{r_i/2}(z^{(i,j)})\cap\Omega} \right),
\]
 we find the no optimal, but simplified, estimate
\begin{eqnarray*}
\|\nabla u\|_{p,Q_T}\leq C(n)
\left[ \|\nabla u\|_{2,Q_T}+ 
\frac{\left(1 +\upsilon \right)^{1/(2+\varepsilon)}}{a_\#}
\times\right. \\ \left. \times
 \left(\sqrt{1+a_\#}\|{\bf f}\|_{2+\varepsilon,Q_T}+ 
\frac{1}{\sqrt{\nu_0}}\|f\|_{2+\varepsilon,Q_T}
+ \sqrt{1+a_\#}K_{{2n/( n+1)}}\|h\|_{2+\varepsilon,\Sigma_T} \right)\right],
\end{eqnarray*}
where
\begin{equation}\label{const}
C(n)=N2^{c(n)}
(r_\#)^{- (n+2)/2}
\left(\frac{2n\beta^{-\varepsilon (n+2)/2}}{4-(n+2)(\upsilon -1)
\varepsilon }\right)^{1/(2+\varepsilon)}.
\end{equation}
Here, $c(n)$ 
 stands for a positive polynomial  function of degree 1 on the space
dimension  $n$.
Therefore, from (\ref{gr2}) we conclude (\ref{cota2}).
 
\section{$W^{1,p}$ regularity ($\ell=2$ and isotropic case)}
\label{sw1p}

In this section, we reformulate the explicit $L^p$-estimate 
of the gradient of a weak solution.
The leading coefficient is assumed
to be $\mathsf{A}=a\mathsf{I}$.

Let us state the following results whose extends to the problem under study
the result obtained in \cite{ben,necas} for the Dirichlet problem.
To this end, we introduce
the Robin-Laplacian operator $\Delta^{\rm R}\in \mathcal{L}( V_p(Q_T);
 \mathcal{W}_p)$ 
and the perturbation $P:
u\in V_p(Q_T)\mapsto Pu \in \mathcal{W}_p$ defined by
\begin{eqnarray*}
\langle -\Delta^{\rm R}u,v\rangle:=
\int_{Q_T}  \nabla u\cdot\nabla v\mathrm{dx}\mathrm{dt}
+\int_{\Sigma_T} u v\mathrm{ds}\mathrm{dt};\\
\langle Pu,v\rangle:=
\int_{Q_T} (1-a) \nabla u\cdot\nabla v\mathrm{dx}\mathrm{dt}
+\int_{\Sigma_T} (1-b(u))u v\mathrm{ds}\mathrm{dt},
\end{eqnarray*}
for all $v\in V_{p'}(Q_T)$.
The term $ b(u)u v$  belongs to $ L^1(\Sigma_T)$  
due to the  embedding
 $L^2(0,T; W^{1,2}(\Omega))\hookrightarrow  L^{2}(\Sigma_T)$,
and the growthness
of $b$ implies that $b(u)\in L^{\infty}(\Sigma_T)$.

The following first result is established.
\begin{prop}\label{l-1}
If $L^{-1}:\mathcal{W}_p\rightarrow  V_p(Q_T)$ is an isomorphism, then
\[
\|L^{-1}Pu\|_{V_p(Q_T)}\leq \|L^{-1}\|_{\mathrm{op}}
\left( (1-a_\#)\|\nabla u\|_{p,Q_T}+(1-b_\#)\|u\|_{2,\Sigma_T}\right),
\]
where $\|L\|_{\rm op}$ stands for the operator norm of $L$.
\end{prop}
\begin{proof}
This property is a consequence of definition of $P$, and
the assumptions (\ref{amin})-(\ref{gama1}) with 
$a_\#,b_\#<1$ and $\ell=2$.
\end{proof}

The existence and uniqueness of weak solutions to the 
linearized variational problem
(\ref{wvf}), {\em i.e.}  $\mathsf{A}=\mathsf{I}$ and $b\equiv 1$,
guarantee that $L=\partial_t-\Delta^{\rm R}$ is an isomorphism
 from $\{w\in D(L): Lw\in \mathcal{W}_p\}$
onto $\mathcal{W}_p$, for any $1<p<\bar p$ and some $\bar p>1$,
 such that $D(L)\subset \mathcal{W}_p\subset R(L) $.
 In particular, this  restriction of $L$ to $\mathcal{W}_p$
 (called the  $\mathcal{W}_p$ realization of the operator $L$)
 satisfies
 \[
\|L^{-1}\|_{\mathrm{op}}\leq\sup_{{\bf f}\in {\bf L}^p(Q_T)\atop
\|{\bf f}\|_{p,Q_T}=1}\sup_{f\in L^p(Q_T)\atop
\|f\|_{p,Q_T}=1}\sup_{h\in L^2(\Sigma_T)\atop
\|h\|_{2,\Sigma_T}=1}\left(\mathcal{M}(1,1)+\mathcal{E}(1,1,p)
\right)|_{\ell=2}:=\Lambda_p,
\]
where $\mathcal{M}(1,1)$ and  $\mathcal{E}(1,1,p)$ are according to (\ref{maa})
and  (\ref{eaa}), respectively.

Therefore, we state the following version of \cite[Thm. 2.2, p. 272]{ben}.
\begin{prop}\label{nablap}
Under the assumptions (\ref{amin})-(\ref{gama1}) with 
$a_\#,b_\#\in ]1-1/\Lambda_p,1[$ and $\ell=2$,
 then any weak solution   $u\in V_{2}(Q_T)$  of (\ref{wvf}) 
 enjoys the following properties
\begin{enumerate}
 \item $u$ satisfies the variational problem $
\langle Lu-Pu -F,v\rangle_{
[V_{p'}(Q_T)]'\times V_{p'}(Q_T)}=0,$
 \item $u$ verifies the following estimate
 \begin{eqnarray*}
(1-\Lambda_p(1-a_\#))\|\nabla u\|_{p,Q_T}+
(1-\Lambda_p(1-b_\#))\|u\|_{2,\Sigma_T}\leq\\
\leq \Lambda_p
\left(\| {\bf f}\|_{p,Q_T}+\|f\|_{p,Q_T}+ \| h\|_{2,\Sigma_T} \right).
\end{eqnarray*}
\end{enumerate}
\end{prop}
\begin{proof}
The point 1 is consequence of the definitions of the operators.
We give an outline of the proof of the point 2.
From the point 1, we have $u=L^{-1}(Pu+F)$. Then, the claimed estimate
follows from Proposition \ref{l-1}.
\end{proof}

Finally, we observe that different explicit estimates are obtained via
the interpolative approach, namely
 the Marcinkiewicz interpolation theorem \cite[pp. 228-230]{gt}.
\begin{thm}
Let $T$ be a linear mapping from $L^q(\Omega)\cap L^r(\Omega)$
into itself, $1\leq q<r<\infty$, and suppose there are constants
$T_1$ and $T_2$ such that
\[
\mu_{Tf}(t)\leq \left({T_1\|f\|_{q,\Omega}\over t}\right)^q,\qquad
\mu_{Tf}(t)\leq \left({T_2\|f\|_{r,\Omega}\over t}\right)^r,
\]
for all $f\in L^q(\Omega)\cap L^r(\Omega)$, and $t>0$.
Then, $T$ extends as a bounded linear mapping from $L^p(\Omega)$ into itself
for any $p$ such that $q<p<r$, and
\begin{equation}
\|Tf\|_{p,\Omega}\leq 2\left({p\over p-q}+{p\over r-p}\right)^{1/p}
T_1^{\alpha}T_2^{1-\alpha}
\|f\|_{p,\Omega}
\end{equation}
holds
for all $f\in L^q(\Omega)\cap L^r(\Omega)$, where $1/p=\alpha/q+(1-\alpha)/r$.
\end{thm}

\section{Steady-state $W^{1,p}$ regularity}
\label{sss}

The higher integrability of the gradient is an useful tool in order to obtain
H\"older continuity (by embedding if $p>n$).
As one knows since long (see e.g. \cite[Ch. 3]{lu} or \cite[Ch. 8]{gt}), 
H\"older continuity can be achieved directly.
When the domain is only Lipschitz,
the coefficients are discontinuous, and  the boundary conditions are mixed,
it is proved in 
\cite{hmrs}, for the most
interesting dimensions $n= 2, 3, 4$. 
For all dimensions it is (unfortunately,
rather implicitly) shown in \cite{grie1,grie2}
 by use of Sobolev-Campanato spaces (which embed for
suitable indices in corresponding H\"older spaces).

An explicit estimate is established in \cite{arxC}.
However, in there the dependence on the data has a wordy expression.
In view of this, 
such estimate is traced back to the celebrated paper by 
Gr\"oger and   Rehberg \cite{grog-reh}
 in the context of elliptic regularity theory for weak solutions
in the case of mixed boundary conditions. In this approach,
it is assumed to be known the upper bound
\begin{equation}\label{defmp}
M_q=\sup\{\|u\|_{1,q,\Omega}:\ u\in W^{1,q}_\Gamma (\Omega),\
\|(-\Delta+I)u\|_{[W^{1,q'}_\Gamma (\Omega)]'}\leq 1\},
\end{equation}
with $W^{1,q}_\Gamma (\Omega)=\{ v\in W^{1,q} (\Omega):\
v=0 \mbox{ on }\Gamma\}$.

Here, we consider the following mixed Neumann-power type problem
to  a linear elliptic equation:

\noindent (NPP) Find $u$ such that verifies, in the sense of distributions,
\begin{eqnarray}
-\nabla\cdot(   \mathsf{A}\nabla u)=
f-\nabla\cdot{\bf f}&\mbox{ in }&\Omega;\label{omega}\\
(\mathsf{A}\nabla u-{\bf f})\cdot{\bf n}=
(h-b(u)u)\chi_\Gamma &\mbox{ on }&\partial
\Omega, \label{gama}
\end{eqnarray}
where $\bf n$ is the unit outward normal to the boundary $\partial\Omega$.
Even more, 
instead of (\ref{defmp}) we set
\[
M_p=\|(-\Delta^{\rm R}_\Omega)^{-1}\|_{\rm op},\]
 where
$\Delta^{\rm R}_\Omega$ is the isomorphism from $W^{1,p}(\Omega)$ onto $[W^{1,p'}(\Omega)]'$ defined by 
\[
\langle -\Delta^{\rm R}_\Omega w,v\rangle =
\int_{\Omega}  \nabla w\cdot
\nabla v \mathrm{dx}+\int_{\Gamma}w v \mathrm{ds}.\]

 We endow the Sobolev space
$ W^{1,p}(\Omega)$ with the norm
\[
\|v\|_{1,p,\Omega}=
\|\nabla v\|_{p,\Omega}+\|v\|_{2,\Gamma}.\]

Although its existence and uniqueness of solutions to (NPP)
 are classical in appropriate subspace
of $H^1(\Omega)$, namely $V_{2,\ell}(\Omega)$, the $W^{1,p}$-regularity $(p>2)$
of the weak solution is a hardship.
Even if the leading coefficient is assumed either to be in VMO \cite{vit}
or to verify a minimal condition \cite{byun}
or if provided by the Laplacian operator, {\em i.e.} $\mathsf{A}=\mathsf{I}$
\cite{ding,wood}, the use of $H^2$-regularity is not allowed since
our right hand side does not belong to a Lebesgue space.

For reader convenience, we exhibit the explicit constant involved in 
the $W^{1,p}(\Omega)$ estimate established in
\cite{arxC}.
\begin{thm}\label{main1}
Let $\Omega$ be a $C^1$ domain,  the assumptions (\ref{amin})-(\ref{gama1})
 be fulfilled,
and  
\[
\upsilon=(8^n+1)2^{6n}
\left[ 
\left(\left( {4a^\#\over a_\#}\right)^{2}+{4+\nu_0\over a_\#} \right)^{1/2}
+1\right]^2,
\] 
where $\nu_0=\nu_0(f)$ is a positive constant if   $f\not=0$,
and $\nu_0(0)=0$ otherwise.
If ${\bf f}\in {\bf L}^{2+\varepsilon}(\Omega)$,
  $f\in L^{2+\varepsilon}(\Omega)$,
 and $h\in L^{2+\varepsilon}(\Gamma)$ for any $\varepsilon \in ]0, 1/(\upsilon-1)
[$, then there exists  a weak solution 
$u \in V_{2,\ell}(\Omega)$ to (\ref{omega})-(\ref{gama}), in the sense
\begin{eqnarray}
\int_{\Omega}    ( \mathsf{A}\nabla u)\cdot
\nabla v \mathrm{dx}+\int_{\Gamma}b(u)u v \mathrm{ds}
=\int_{\Omega}{\bf f}\cdot\nabla v \mathrm{dx}
+\nonumber\\
+\int_{\Omega}fv \mathrm{dx}
+\int_{\Gamma} h v\mathrm{ds},\quad\forall v\in V_{2,\ell}(\Omega).
\label{pbu}
\end{eqnarray}
such that belongs to $W^{1,2+\varepsilon}(\Omega)$.
In particular,
\begin{eqnarray}
\| \nabla u\|_{2+\varepsilon,\Omega}^{2+\varepsilon}
\leq \frac{2^{n(1+\varepsilon/2)}N}{4-(n+2)(\upsilon-1)\varepsilon}\left[
\left(8\over (r_\#)^n\right)^{\varepsilon/2} 4
\| \nabla u\|_{2, \Omega}^{2+\varepsilon} +\right.\nonumber\\
\left. +\left( 2^{2+3\varepsilon/2}+
 \upsilon(4+(n+2)\varepsilon)\right)
\|  {\mathcal F}\|_{2+\varepsilon, \Omega}^{2+\varepsilon}
+ \left( 4+ \upsilon(4+(n+2)\varepsilon)\right)
 \| {\mathcal H}\|_{2+\varepsilon,\Gamma}^{2+\varepsilon}\right]
,\quad\label{cotam1}
\end{eqnarray}
where 
 \begin{eqnarray*}
 {\mathcal F}&= &(a_\#)^{-1/2}\left[\left(\frac{2}{ a_\#}+2\right)|{\bf f}|^2+
\frac{1}{\nu_0} |f|^2\right]^{1/2}
;\\ {\mathcal H}&= &2
{\sqrt{2+2^{-1/n}a_\#}\over a_\#} K_{2n/(n+1)} |h|,
\end{eqnarray*}
with $r_\#>0$
and $N\in\mathbb{N}$ being dependent on the space dimension.
\end{thm}

Let us extend the existence result for the
mixed Dirichlet-Neumann problem  \cite[Theorem 1]{grog-reh} to the following 
one for the mixed Neumann-power type problem (NPP).
\begin{prop}\label{ppv}
Suppose $\ell=2$. 
If  $\mathsf{A}$ is symmetric,  ${\bf f}\in {\bf L}^p(\Omega)$,
$f\in L^{pn/(p+n)}(\Omega)$, and $h\in L^2(\Gamma)$, with  $p>2$ such that
\begin{equation}\label{defvk}
a^\#/M_p>\varkappa:=
\max\{\sqrt{(a^\#)^2-(a_\#)^2},
| a^\#-a_\#b_\#/a^\#|\},
\end{equation} 
then the weak solution $u\in V_{2,\ell}(\Omega)$ of (\ref{pbu}) satisfies
\begin{equation}\label{el2}
\|\nabla u\|_{p,\Omega}+\|u\|_{2,\Gamma}\leq {M_p a_\#\over a^\#-\varkappa
M_p}(
\| {\bf f}\|_{p,\Omega}+S_{p'}\|f\|_{pn/(p+n),\Omega}+
\| h\|_{2,\Gamma}),
\end{equation}
where $S_{p'}$ is according to Remark \ref{rsp}.
\end{prop}
\begin{proof} 
The monotone theory for elliptic equations (see for instance
\cite[Corollary 2.2, p. 39]{show}) ensures the existence of 
 $u \in V_{2,\ell}(\Omega)$  solving 
\begin{eqnarray}
\int_{\Omega}  \nabla u\cdot
\nabla v \mathrm{dx}+\int_{\Gamma}b(u)u v \mathrm{ds}
=\int_{\Omega}{\bf F}\cdot\nabla v \mathrm{dx}
+t\int_{\Omega}fv \mathrm{dx}+\nonumber\\
+\int_{\Gamma}Gv \mathrm{ds},
 \quad\forall v\in V_{2,\ell}(\Omega),\label{pbua}
\end{eqnarray}
with
\begin{eqnarray*}
{\bf F}&=&  ( \mathsf{I}-   t\mathsf{A})\nabla u+t{\bf f} ;\\
G&=&(1-tb(u))u+th,
\end{eqnarray*}
for any $t>0$.

We seek for a unique fixed point of the continuous linear mapping $Q:
W^{1,p}(\Omega)\rightarrow W^{1,p}(\Omega)$ defined by 
$Qw=(-\Delta^{\rm R}_\Omega)^{-1}(L_tw)$, with
\[ 
\langle L_tw,v\rangle =
\int_{\Omega}   [( \mathsf{I}-   t\mathsf{A}) \nabla w+t{\bf f} ]\cdot
\nabla v \mathrm{dx}
+t\int_{\Omega}fv \mathrm{dx}+\int_{\Gamma}[ (1-tb(u))w+th] v \mathrm{ds}, 
\] 
for all $v\in W^{1,p'}(\Omega)$.

The existence of a unique fixed point is guaranteed 
if $Q$ is strictly contractive.
Let $u_1$, $u_2\in W^{1,p}(\Omega)$ be arbitrary, then
\[
\|Qu_1-Qu_2\|_{1,p,\Omega}\leq M_p\left(\|
  ( \mathsf{I}-   t\mathsf{A})\nabla (u_1-u_2)\|_{p,\Omega}+
  \|(1-tb(u))(u_1-u_2)\|_{2,\Gamma}
  \right).
\]

For all ${\bf y}\in {\bf L}^p(\Omega)$, we have the relation \cite{grog-reh}
\begin{equation}\label{ay}
\|(\mathsf{I}-{a_\#\over (a^\#)^2}\mathsf{A}){\bf y}\|_{p,\Omega}\leq
\sqrt{1-(a_\#/a^\#)^2}\|{\bf y}\|_{p,\Omega}.
\end{equation}
Letting $t=a_\#(a^\#)^{-2}$, gathering the two above inequalities we obtain
 \begin{eqnarray*}
\|Qu_1-Qu_2\|_{1,p,\Omega}\leq M_p\left(\sqrt{1-(a_\#/a^\#)^2}\|
 \nabla (u_1-u_2)\|_{p,\Omega} +\right. \\ \left.
+ |1-{a_\#b_\#/ (a^\#)^2}| \|u_1-u_2\|_{2,\Gamma}
  \right).
\end{eqnarray*}
By (\ref{defvk}), $Q$ is a strict contraction, and then there exists
$w\in W^{1,p}(\Omega)$ such that $w=(-\Delta^{\rm R}_\Omega)^{-1}(L_tw)$.
By uniqueness of solution in $V_{2,\ell}(\Omega)$, then $w\equiv u$ 
verifies
\[
\| u\|_{1,p,\Omega}\leq {M_p \over a^\#}\left(
\varkappa \| u\|_{1,p,\Omega}+{a_\#\over a^\#}(
\| {\bf f}\|_{p,\Omega}+S_{p'}\|f\|_{pn/(p+n),\Omega}+
\| h\|_{2,\Gamma})
\right),
\]
since $p'<2$,
which implies (\ref{el2}).
\end{proof}

\begin{rem}
The choice of the involved constant in (\ref{defvk}), which comes from
(\ref{ay}), is not optimal. In the work
\cite{necas} the author shows that if there exists $\theta\in [0,1[$
such that $\mathsf{A}$ verifies 
\[\sum_{i,j=1}^n
{1\over 2}(A_{ij}-A_{ji})\xi_i\eta_j\leq a_\#\theta
\left(\sum_{i=1}^n\xi_i^2\right)^{1/2}
\left(\sum_{j=1}^n\eta_j^2\right)^{1/2}
\]
then
\[
\|(\mathsf{I}-{1\over a^\#}\mathsf{A}){\bf y}\|_{p,\Omega}\leq
c(1-(1-\theta) a_\#/a^\#)\|{\bf y}\|_{p,\Omega},\quad\forall 
{\bf y}\in {\bf L}^p(\Omega),
\]
where $c>1$ is dependent on $p$ $(p\geq 2)$
and the space dimension $n$ as follows
\[
\left(\sum_{i=1}^n |y_i|^2\right)^{p/2}\leq c^p\sum_{i=1}^n |y_i|^p
.\]
In particular, $c=2^{1/2-1/p}$ if $n=2$.
 For a symmetric $\mathsf{A}$, we emphasize that $\theta=0$, and
\[ 
c(1-a_\#/a^\#)\leq \sqrt{1-(a_\#/a^\#)^2}\quad\mbox{if } a_\#/a^\#
\geq {c^2-1\over c^2+1}.
\] 
Moreover, if instead (\ref{defvk}) we suppose
\[
M_p\max\{c(1-a_\#/a^\#),|1-a_\#b_\#/(a^\#)^2|\}
\leq 1-a_\#/(2a^\#),
\]
then (\ref{el2}) reads
\[\|\nabla u\|_{p,\Omega}+\|u\|_{2,\Gamma}\leq \frac{2M_p}{a_\#}(
\| {\bf f}\|_{p,\Omega}+S_{p'}\|f\|_{pn/(p+n),\Omega}+
\| h\|_{2,\Gamma}).
\]
\end{rem}

 \begin{rem}
 We emphasize that Proposition \ref{ppv} does not contradicts the counterexample
 of the existence of a function $u\in H^1(\mathbb{R}^2)$ solving
 the elliptic equation $\nabla\cdot(   \mathsf{A}
\nabla u)=0$ in $\mathbb{R}^2$ such that does not belong to $W^{1,p}(\Omega)$
for some $p>2$.
\end{rem}

\end{document}